\documentclass[10pt, notitlepage]{amsart}
\usepackage[utf8]{inputenc}
\usepackage[english]{babel}
\usepackage[a4paper, total={6in, 10in}]{geometry}
\usepackage{setspace}
\usepackage{tikz-cd}
\usepackage{amsmath}
\usepackage{amssymb}
\usepackage{amsthm}
\usepackage{amsfonts}
\usepackage{bbm}
\usepackage{float}
\usepackage{braket}
\usepackage{multicol}
\usepackage{mathtools}
\usepackage{url}
\usepackage{hyperref}

\theoremstyle{definition}
\newtheorem{theorem}{Theorem}[section]
\newtheorem{ltheorem}{Theorem}[section]

\newtheorem{lcorollary}[ltheorem]{Corollary}

\newtheorem{corollary}[theorem]{Corollary}
\newtheorem{lemma}[theorem]{Lemma}
\newtheorem{definition}[theorem]{Definition}

\newtheorem{proposition}[theorem]{Proposition}
\newtheorem{proposition-definition}[theorem]{Proposition-Definition}
\newtheorem{remark}[theorem]{Remark}
\newtheorem{example}[theorem]{Example}
\newtheorem{question}[theorem]{Question}

\makeatletter
\renewcommand\tableofcontents{%
	\null\hfill\textbf{\Large\contentsname}\hfill\null\par
	\@mkboth{\MakeUppercase\contentsname}{\MakeUppercase\contentsname}%
	\@starttoc{toc}%
}
\makeatother

\title{Categorical approach to rigidity of Roe-like algebras of coarse spaces}
\author{K. Krutoy}
\date{\today}
\email{krutoy@imj-prg.fr}
\address{Institut de Math\'ematiques de Jussieu - Paris Rive Gauche (IMJ- PRG) \\ Universit\'e Paris Cit\'e, B\^atiment Sophie Germain, 8 Place Aur\'elie Nemours, 75013 Paris, France}

\begin{document}
	\maketitle
	\begin{abstract}
		We demonstrate that any full and faithful $*$-functor between approximable categories of locally finite coarse spaces induces a coarse embedding between the underlying spaces. Furthermore, we establish a general characterisation of such $*$-functors between approximable categories and prove that the functor associating each locally finite coarse space with its approximable category is full and faithful.
	\end{abstract}
	\section*{Introduction}
	Large-scale geometry studies metric spaces by identifying those that exhibit similar properties when observed at large distances. Specifically, given two metric spaces $(X,d)$ and $(Y, \partial)$, a map $f \colon X \to Y$ is said to be \emph{coarse} if, for every $R \geq 0$, there exists $S \geq 0$ such that
	$$
		\partial(f(x), f(y)) \leq S \quad \text{whenever} \quad d(x,y) \leq R.
	$$
	Coarse maps are precisely those that preserve large-scale geometry. We identify coarse maps that are uniformly close, that is, maps $f,g \colon X \to Y$ satisfying $\sup_{x \in X} \partial(f(x), g(x)) < \infty$. Such maps are called \emph{close}. Consequently, we obtain a category whose objects are metric spaces and whose morphisms are the closeness classes of coarse maps. The classes of coarse maps $f \colon X \to Y$ for which there exist a coarse map $g \colon Y \to X$ satisfying
	$$
		f \circ g \sim \operatorname{id}_Y, \text{ and } \enspace g \circ f \sim \operatorname{id}_X
	$$
	constitute the isomorphisms in this category. These maps are referred to as \emph{coarse equivalences}, and metric spaces related by such maps are said to be \emph{coarsely equivalent}. 
	
	An active research direction in large-scale geometry is geometric group theory, which aims to deduce algebraic properties of groups from the large-scale geometry of their Cayley graphs. Specifically, geometric group theory investigates which algebraic properties of a group are preserved when passing to another group whose Cayley graph is coarsely equivalent to that of the original group. One such property is \emph{virtual nilpotence}, which is equivalent to the polynomial growth of balls in the respective Cayley graph, as demonstrated in the renowned article by Gromov \cite{PMIHES_1981__53__53_0}. Another area within large-scale geometry involves the study of index theory for elliptic operators on Riemannian manifolds, as illustrated, for instance, in \cite{RoeCCaIToCRM}.
	
	\subsection*{Operator algebraic methods in coarse geometry}\
	
	In \cite[page 262]{Niblo1993GeometricGT}, Gromov introduced the translation algebras $T(X)$ of a metric space $(X,d)$. These algebras consist of endomorphisms of a complex vector space that act on a fixed basis in a ``controlled manner''. Let $V$ be a complex vector space with a fixed basis $\mathfrak{b}$. A geometric $\mathbb{C}$-module over $(X,d)$ is a triple $(V, \mathfrak{b}, \phi)$, where
	$$
		\phi \colon \mathfrak{b} \to X
	$$
	is a map such that the preimages of bounded subsets of $X$ are finite. Every endomorphism $T$ of $V$ can be represented as a complex $\mathfrak{b} \times \mathfrak{b}$-matrix $[T_{i,j}]_{i,j \in \mathfrak{b}}$, where each column contains only finitely many nonzero entries. The propagation of $T$ is defined as
	$$
		\operatorname{prop}(T) = \sup \{d(\phi(b_1),\phi(b_2)) \mid (b_1, b_2) \in \mathfrak{b} \times \mathfrak{b}, \enspace T_{b_1, b_2} \neq 0 \}.
	$$
	Endomorphisms of finite propagation can be viewed as matrices whose nonzero entries lie within a tube of fixed width around the diagonal of $\mathfrak{b} \times \mathfrak{b}$. To illustrate this, consider the example $X = \mathbb{Z}$ equipped with the Euclidean metric. Consider a geometric $\mathbb{C}$-module that consists of a complex vector space whose basis is indexed by $\mathbb{Z}$, equipped with a map $\phi$ that assigns to each basis element its index. Here, elements of propagation $0$ are precisely the diagonal matrices. Elements of propagation $\le 1$ correspond to matrices whose nonzero entries are confined within a strip between the first sub-diagonal and the first super-diagonal. In the diagram below, the intensity of the shading represents the magnitude of propagation.
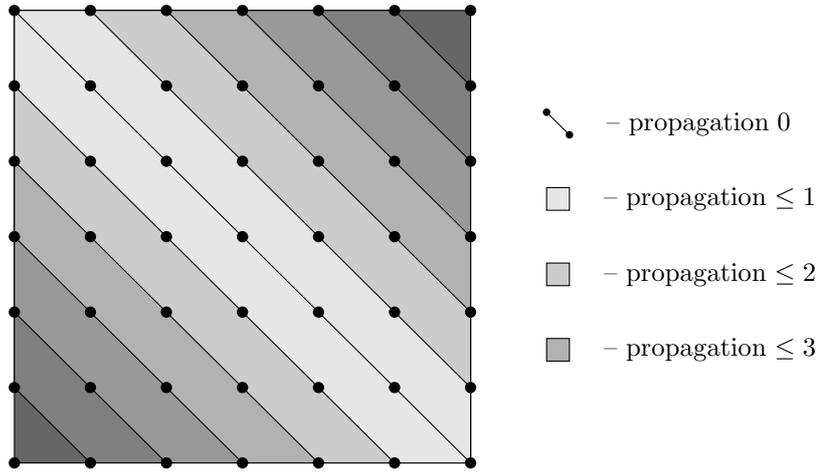
\vspace{0.1cm}
$$
	\begin{tikzpicture}[scale=1] \label{Figure 1}
		
		\node at (1,0.8)[]{};
		
		\draw[fill=black!60, draw opacity=0] (1, 7) -- (1,1) -- (1,1) -- (7,1) -- (7,7) -- (7, 7) -- cycle;
		\draw[fill=black!50, draw opacity=0] (1, 7) -- (1,2) -- (2,1) -- (7,1) -- (7,6) -- (6, 7) -- cycle;
		\draw[fill=black!40, draw opacity=0] (1, 7) -- (1,3) -- (3,1) -- (7,1) -- (7,5) -- (5, 7) -- cycle;
		\draw[fill=black!30, draw opacity=0] (1, 7) -- (1,4) -- (4,1) -- (7,1) -- (7,4) -- (4, 7) -- cycle;
    	\draw[fill=black!20, draw opacity=0] (1, 7) -- (1,5) -- (5,1) -- (7,1) -- (7,3) -- (3, 7) -- cycle;
    	\draw[fill=black!10, draw opacity=0] (1, 7) -- (1,6) -- (6,1) -- (7,1) -- (7,2) -- (2, 7) -- cycle;
    
    	\foreach \x in {1,...,7} \foreach \y in {1,...,7} {
    		\node at (\x,\y)[circle,fill,inner sep=1.5pt]{};
    	}
    
    	\foreach \x in {1,...,6} {
    		\draw (\x,8-\x) -- (\x+1, 8-\x-1);
    	}
  		
  		\node at (8,5.65)[circle,fill,inner sep=1pt]{};
  		\node at (8.3,5.35)[circle,fill,inner sep=1pt]{};
  		\draw (8, 5.65) -- (8.3, 5.35);
  		\node[draw opacity=0] at (10,5.50) { -- propagation 0};
  		
  		\path[draw=black, fill=black!10] (8,4.35) rectangle (8.3,4.65);
  		\node[draw opacity=0] at (10.15,4.50) { -- propagation $\le1$};
  		\path[draw=black, fill=black!20] (8,3.35) rectangle (8.3,3.65);
  		\node[draw opacity=0] at (10.15,3.50) { -- propagation $\le2$};
  		\path[draw=black, fill=black!30] (8,2.35) rectangle (8.3,2.65);
  		\node[draw opacity=0] at (10.15,2.50) { -- propagation $\le3$};
  	\end{tikzpicture}
  $$
 \vspace{0.1cm}
	
	The Gromov translation algebra $T(X)$ of a geometric $\mathbb{C}$-module is a subalgebra of finite propagation endomorphisms of the underlying vector space. These algebras have proven to be instrumental in the study of group amenability \cite{Elek1997TheO} and various other coarse invariants \cite[Chapter 4]{RoeBook}. Notably, these structures extend to a much broader setting; in particular, translation algebras can be defined over arbitrary rings. For further details, we refer the reader to \cite[Chapter 4]{RoeBook}.
	
	Roe introduced Roe algebras \cite{RoeCCaIToCRM} to advance the study of index theory for Dirac-type operators on complete Riemannian manifolds. These algebras serve as analytic analogues of Gromov translation algebras. Specifically, given a metric space $(X,d)$ and a nondegenerate representation $\pi \colon C_0(X) \to \mathcal{B}(H)$ of the algebra of continuous functions vanishing at infinity on a separable, infinite-dimensional Hilbert space, one defines a concrete $C^*$-subalgebra of $\mathcal{B}(H)$ generated by operators $T \in \mathcal{B}(H)$ that satisfy the following conditions:
	\begin{enumerate}
		\item (\emph{Local compactness}) For every compact subset $K \subset X$ the operators $\mathbbm{1}_K T$ and $T \mathbbm{1}_K$ are compact in $\mathcal{B}(H)$;
		\item (\emph{Finite propagation}) The quantity $\operatorname{prop}(T) = \sup \{d(x,y) \mid (x,y) \in \operatorname{Supp}(T)\}$ is finite, where $\operatorname{Supp}(T)$ is a subset of $X \times X$ consisting of points $(x,y)$ such that for all open neighbourhoods $U$ of $x$ and $V$ of $y$ the quantity $\|\mathbbm{1}_U T \mathbbm{1}_V\|$ is nonzero.
	\end{enumerate}
	Here, the nondegenerate representation of $C_0(X)$ serves as the analogue of a geometric $\mathbb{C}$-module. The finite propagation condition directly corresponds to the finite propagation condition of endomorphisms of $\mathbb{C}$-modules, while the local compactness condition replaces the requirement that the preimages of bounded sets under $\phi$ are finite. Notable examples of Roe algebras are uniform Roe algebras. Given a uniformly locally finite metric space $(X,d)$, one may define a Roe algebra associated with the representation of $\ell^{\infty}(X)$ on $\ell^2(X)$ via left multiplication. The resulting algebra is known as the uniform Roe algebra of $(X,d)$. The significance of Roe algebras has been demonstrated in various contexts, including the construction of counterexamples to the Baum--Connes conjecture for groupoid $C^*$-algebras \cite{HLS_CtBCC} and the study of positive scalar curvature on Riemannian manifolds \cite{seyedhosseini2022variant, WillettYu2020}. Furthermore, certain variants of these algebras have been employed as models for topological phases of disordered materials \cite{EM_physics}.
	
	\subsection*{Rigidity problem for Roe algebras}\
	
	It is well-known \cite[Chapter 5]{WillettYu2020} that for two coarsely equivalent metric spaces $(X, d)$ and $(Y, \partial)$ certain Roe algebras associated to $(X,d)$ and $(Y, \partial)$ are $*$-isomorphic. The rigidity problem for Roe algebras concerns whether $*$-isomorphisms of Roe algebras imply the existence of coarse equivalences between the underlying metric spaces. $\check{\text{S}}$pakula and Willett \cite{SWrigidity} demonstrated that this holds for all uniformly locally finite metric spaces with property~A. Subsequent developments were presented in \cite{BFVemb2019, BV_GelfandTypeDuality, BF_rigidity, BVFcoronas}, culminating in a definitive result in \cite{BFBFKVW_rigidity}. Specifically, the authors of \cite{BFBFKVW_rigidity} established that the uniform Roe algebras of uniformly locally finite spaces are Morita equivalent if and only if the underlying metric spaces are coarsely equivalent. Concurrently, the question of whether an isomorphism of uniform Roe algebras induces a bijective coarse equivalence was solved for property~A spaces in \cite{WW_CsA}, for nonamenable spaces in \cite{BFBFKVW_rigidity}, and for expander graphs in \cite{BBFVW_BijCoaExpanders}. However, the general case remains an open problem.
	
	A recent sequence of papers \cite{MVmodules, MVrigidity, MVpaper} by D. Martínez and F. Vigolo reformulated the theory of Roe algebras within the framework of general coarse spaces and coarse modules. In \cite{MVmodules}, the authors extended Roe's notion of a coarse module \cite{RoeCCaIToCRM} to encompass spaces not necessarily endowed with a topology. They reconstructed numerous results concerning the structure of Roe algebras in this new framework. Their second paper \cite{MVrigidity} focuses on the rigidity problem for Roe-like algebras. Specifically, Section 5 generalises the uniformisation theorem from \cite{BF_rigidity} to apply to general coarse modules over countably generated coarse spaces. They introduced approximating relations associated with a map between Roe-like algebras, a family of relations between coarse spaces dependent on specific parameters. They demonstrated that for any $*$-isomorphism (or Morita equivalence) between Roe-like algebras of coarse modules over countably generated coarse spaces, parameters can be chosen such that the approximating relations become coarse equivalences. The theory developed in these papers extends well beyond the context of Roe algebras, encompassing cases such as approximable and quasi-local algebras. 
	
	\subsection*{Main results}\
	
	The papers \cite{MVmodules, MVrigidity} provide a solution to the rigidity problem for Roe algebras of faithful\footnote{See Section 1 for preliminaries on the subject.} coarse modules over countably generated coarse spaces. However, the relationship between $*$-isomorphisms of Roe algebras of coarse modules and the induced coarse equivalences between the underlying coarse spaces remains not entirely transparent. In particular, we seek to understand how the induced coarse equivalence depends on the choice of faithful coarse modules. Furthermore, it is of significant interest to establish a functorial connection between the category of coarse spaces and a suitable category arising from operator algebraic techniques, which encapsulates the rigidity result from \cite{MVrigidity}.
	
	This paper examines the compatibility of coarse equivalences induced by $*$-isomorphisms of several faithful coarse modules over a countably generated locally finite coarse space. Particularlly, let $\mathcal{X}$ and $\mathcal{Y}$ be two locally finite coarse spaces, $\{C_X^i\}_{i \in I}$ a set of coarse $\mathcal{X}$-modules, $\{C_Y^i\}_{i \in I}$ a set of coarse $\mathcal{Y}$-modules, and $\phi_i$ a $*$-isomorphism between the algebras of approximable operators associated with $C^i_X$ and $C^i_Y$ for every $i \in I$. We aim to determine the conditions under which the induced coarse equivalences $f_i \colon \mathcal{X} \to \mathcal{Y}$ coincide. 
	
	To address this question, we assemble the coarse modules over a locally finite coarse space $\mathcal{X}$ into a $C^*$-category $\mathcal{A}(\mathcal{X})$, analogous to the one studied in \cite{BunkeEngel_HomThwBornCoarseSp} in the context of the coarse Baum--Connes conjecture. Subsequently, we investigate full and faithful $*$-functors between these categories. The following theorem demonstrates that full and faithful $*$-functors induce coarse embeddings\footnote{See Definition \ref{Def coarse maps}} between the underlying spaces.
	
	\begin{ltheorem}[Theorem \ref{*-functor => coarse emb} and Corollary \ref{ff *-funct => coarse emb (2)}] \label{theorem A}
		Let $\mathcal{X}$ and $\mathcal{Y}$ be locally finite countably generated coarse spaces with at most countably many connected components and $F \colon \mathcal{A}(\mathcal{X}) \to \mathcal{A}(\mathcal{Y})$ be a full and faithful $*$-functor. Then $F$ induces a coarse embedding $f^F \colon \mathcal{X} \to \mathcal{Y}$
	\end{ltheorem}
	
	We then examine the structure of full and faithful $*$-functors between the associated categories and obtain a complete characterisation of these functors. Specifically, we introduce a congruence relation $\cong_u$ on the category $\mathcal{A}(\mathcal{X})$ and demonstrate that any full and faithful $*$-functor is equivalent to one induced by a coarse embedding when viewed as a functor between the quotient categories.
	
	\begin{ltheorem}(Theorem \ref{theorem: Structure_of_faf_star_functors}) \label{theorem B}
		Let $\mathcal{X}$ and $\mathcal{Y}$ be locally finite countably generated coarse spaces and $F \colon \mathcal{A}(\mathcal{X}) \to \mathcal{A}(\mathcal{Y})$ be a full and faithful $*$-functor. Then, it is naturally isomorphic to a $*$-functor induced by a coarse equivalence $f \colon \mathcal{X} \to \mathcal{Y}$ as functors between the quotient categories.
	\end{ltheorem}
	
	We conclude the article with several corollaries and open questions. In particular, we provide a solution to our initial problem.
	
	\begin{lcorollary}[Corollary \ref{corollary: initial goal}] \label{corollary A}
		Suppose given two locally finite countably generated coarse spaces $\mathcal{X}$ and $\mathcal{Y}$, a family of faithful coarse $\mathcal{X}$-modules $\{C_i^X\}_{i \in I}$, a family of faithful coarse $\mathcal{Y}$-modules $\{C_i^Y\}_{i \in I}$, and, for every $i \in I$, a $*$-isomorphism $\phi_i$ between the associated algebras of approximable operators of $C_i^X$ and $C_i^Y$. The following are equivalent:
		\begin{enumerate}
			\item The induced coarse equivalences $f_i \colon \mathcal{X} \to \mathcal{Y}$ coincide up to closeness;
			\item There exists a full and faithful $*$-functor $F \colon \mathcal{A}(\mathcal{X}) \to \mathcal{A}(\mathcal{Y})$ such that the isomorphisms $F \colon \operatorname{End}_{\mathcal{A}(\mathcal{X})}(C^X_i) \to \operatorname{End}_{\mathcal{A}(\mathcal{Y})}(C^Y_i)$ coincide with $\phi_i$.
		\end{enumerate}
	\end{lcorollary}
	
	\begin{lcorollary}[Corollary \ref{corollary: cat equiv}] \label{corollary B}
		The functor $\mathcal{A}$ from the category of locally finite coarse spaces and equivalence classes of coarse embeddings modulo closeness to the category $\mathcal{C} / \cong_u$ of approximable $C^*$-categories and classes of natural modulo $\cong_u$ isomorphisms of full and faithful $*$-functors is full and faithful.
	\end{lcorollary}
	
	The authors of \cite{MVrigidity, MVmodules, MVpaper} investigated the relationships between $*$-isomorphisms of Roe-like algebras and the geometry of the underlying coarse spaces. We explore the relationships between collections of $*$-isomorphisms of Roe-like algebras and the geometry of their associated coarse spaces. The core elements of the proofs of Theorems \ref{theorem A} and \ref{theorem B} include assembling coarse modules into additive $C^*$-categories (see Section 1.3) and introducing a family of invariants of coarse modules, specifically the domains of coarse modules (see Definition \ref{definition:: domains}). The invariance of the domains of coarse modules under invertible approximable operators (see Theorem \ref{theorem: domain invariance}) enables the application of more category-theoretic methods (), as demonstrated in Sections 3 and 4.
	
	The manuscript is structured as follows. In Section 1, we recall the fundamental notions of coarse spaces, including the relational perspective on coarse maps (see Section 1.1). Additionally, we introduce the concept of a \emph{locally finite coarse measurable (LFCM) space} (see Definition \ref{definition: discrete space}), defined as a coarse space equipped with a suitable measurable structure. We demonstrate that the category of LFCM spaces, with equivalence classes of coarse measurable maps, is equivalent to the category of locally finite coarse spaces. The notion of LFCM spaces is motivated by the need for techniques enabling work directly with a coarse space without requiring its discretisation. For instance, every proper metric space equipped with a Boolean algebra of Borel sets constitutes a LFCM space. We adapt the theory of coarse modules from \cite{MVmodules} to the framework of LFCM spaces (see Section 1.2). The principal advantage of this approach is that all coarse modules over a LFCM space are discrete in the sense of \cite[Definition 4.18]{MVmodules}. In Section 1.3, we introduce the approximable category $\mathcal{A}(\mathcal{X})$ for a LFCM space $\mathcal{X}$ and several variations of it. We establish that $\mathcal{A}(\mathcal{X})$ is an additive $C^*$-category and demonstrate that $\mathcal{A}(\mathcal{X})$ depends functorially on $\mathcal{X}$. In Section 1.4, we briefly review the strategy behind the proof of the rigidity theorem in \cite{MVrigidity} for algebras of approximable operators. In Section 2, we investigate various invariants of coarse modules, particularly their domains, which are specific parts of the underlying LFCM space whose geometry is reflected in the module. We prove that isomorphic modules have asymptotic domains (see Theorem \ref{theorem: domain invariance}). We also derive a nonfaithful analogue of \cite[Corollary 6.20]{MVrigidity} (see Corollary \ref{corollary: Nonfaithful rigidity}). Section 3 is dedicated to the proof of Theorem \ref{theorem A} for different instances of the approximable category, while Section 4 focuses on the proof of Theorem \ref{theorem B}. Both theorems rely on the additivity of approximable categories and the invariance of domains under isomorphisms of coarse modules as their central arguments. Finally, in Section 7, we proved Corollaries \ref{corollary A} and \ref{corollary B} and posed several open questions concerning other variants of categories of coarse modules.
	
	\begin{center}
		\textbf{Acknowledgements}
	\end{center}
	
	The author sincerely thanks his advisors, Alessandro Vignati and Romain Tessera, for their insightful discussions on the topic and their valuable suggestions for improving the presentation. We are also deeply grateful to Diego Martínez and Federico Vigolo for identifying a mistake in Theorem 2.8 in the earlier version of the manuscript and for their valuable comments and remarks.
	
	\section{Preliminaries}
	
	In this section, we present the preliminary results on coarse geometry, coarse geometric modules and the strategy of the proof of the main result of \cite{MVrigidity}. The theory of coarse modules and their categories have undergone minor changes to facilitate subsequent results. 
	 
	\subsection{Coarse geometry}
	
	We start by recalling the definition of a coarse space. We refer the reader to the second chapter of \cite{RoeBook} for an in-depth discussion of the subject.  Broadly speaking, coarse spaces serve as abstractions of metric spaces, allowing the study of large-scale geometric properties.
	
	\begin{definition}
    	A \emph{coarse space} is defined as a pair $(X, \mathcal{E}^X)$, where $X$ is a set, and $\mathcal{E}^X \subset \mathcal{P}(X \times X)$ satisfies the following conditions:
    	\begin{enumerate}
        	\item $\Delta_X = \{(x,x) \mid x \in X\} \in \mathcal{E}^X$, referred to as the diagonal;
        	\item If $E \in \mathcal{E}^X$ and $F \subseteq E$, then $F \in \mathcal{E}^X$;
        	\item If $E, F \in \mathcal{E}^X$, then $E \cup F \in \mathcal{E}^X$;
        	\item If $E \in \mathcal{E}^X$, then $E^T = \{(y,x) \mid (x,y) \in E\} \in \mathcal{E}^X$;
        	\item If $E, F \in \mathcal{E}^X$, then $F \circ E = \{(z,x) \mid \exists y \in X \colon (z,y) \in F, (y,x) \in E\} \in \mathcal{E}^X$.
    	\end{enumerate}
    	A coarse space $(X, \mathcal{E})$ is said to be \emph{countably generated} if there exists a countable subset $\{E_k\}_{k \in \mathbb{N}} \subset \mathcal{E}$ such that $\mathcal{E}$ is the smallest coarse structure on $X$ containing $\{E_k\}_{k \in \mathbb{N}}$.
	\end{definition}

	The elements of $\mathcal{E}^X$ are typically referred to as \emph{controlled sets} or \emph{entourages}. An entourage $E$ is \emph{symmetric} if $E^T = E$. A symmetric entourage that contains the diagonal $\Delta_X$ is said to be a \emph{gauge}. For $A \subset X$ and an entourage $E \in \mathcal{E}^X$, the \emph{$E$-neighbourhood} of $A$ is the set
	$$
		E[A] = \{x \in X \mid \exists a \in A \colon (x, a) \in E\}.
	$$ 
	A subset $A$ of $X$ is said to be \emph{$E$-bounded} if contained within the $E$-neighbourhood of any of its points; equivalently, if $A \times A \subset E$. A subset is \emph{bounded} if it is \emph{$E$-bounded} for some $E \in \mathcal{E}^X$.
	
	Recall from \cite[Chapter 2]{RoeBook} that an extended metric space $(X,d)$ can be regarded as a coarse space by equipping $X$ with the coarse structure generated by the family $\{E_n\}_{n \in \mathbb{N}}$, where  
	$$  
		E_n = \{(x,y) \mid d(x,y) \leq n\}.  
	$$  
 	By \cite[Theorem 2.55]{RoeBook}, a coarse structure on a set $X$ is countably generated if and only if it is induced by an extended metric on $X$. For an example of a coarse structure that is not countably generated, see \cite[Section 2.3]{BFV_GURAR}.
	
	\begin{definition}\label{Def coarse maps}
    Let $(X, \mathcal{E}^X)$ and $(Y, \mathcal{E}^Y)$ be coarse spaces.
    	\begin{enumerate}
        	\item A map $f \colon X \to Y$ is \emph{coarse} if, for every entourage $E \in \mathcal{E}^X$, the set $(f \times f)(E)$ is an entourage in $Y$.
        	\item Two coarse maps $f,g \colon X \to Y$ are \emph{close} (denoted $f \sim g$) if, for every entourage $E \in \mathcal{E}^X$, the set $(f \times g)(E)$ is an entourage in $Y$.
        	\item A coarse map $f \colon X \to Y$ is a \emph{coarse equivalence} if there exists a map $g \colon Y \to X$ such that $f \circ g \sim \operatorname{id}_Y$ and $g \circ f \sim \operatorname{id}_X$.
        	\item A coarse map $f \colon X \to Y$ is a \emph{coarse embedding} if it is a coarse equivalence as a map from $X$ to $\operatorname{im}(f)$, where $\operatorname{im}(f)$ is equipped with the coarse structure 
        	$$
            	\mathcal{E}^{\operatorname{im}(f)} = \{E \cap (\operatorname{im}(f) \times \operatorname{im}(f)) \mid E \in \mathcal{E}^Y\}.
        	$$
    	\end{enumerate}
	\end{definition}
	
	In the metric setting, the definitions of coarse maps and the closeness relation coincide with those given in the introduction. It can be verified that closeness is an equivalence relation on the set of coarse maps. Morphisms of coarse spaces should be regarded as closeness classes of coarse maps, in which case, coarse equivalences correspond precisely to isomorphisms. Note that different authors may adopt distinct and possibly inequivalent definitions of coarse maps.
	
	\begin{definition}[\cite{MVmodules}, Definition 3.11]
		Let $(X, \mathcal{E})$ be a coarse space and $A, B \subset X$. We say that $A$ is \emph{subordinate} to $B$ ($A \prec B$) if there exists an entourage $E \in \mathcal{E}^X$ such that $A \subset E[B]$. The set $A$ is \emph{asymptotically equivalent} to $B$ ($A \asymp B$) if both $A \prec B$ and $B \prec A$ hold.
	\end{definition}
	
	For two coarse spaces $(X, \mathcal{E}^X)$ and $(Y, \mathcal{E}^Y)$, we endow their product $Y \times X$ with the minimal coarse structure generated by $\mathcal{E}^Y \times \mathcal{E}^X$. In this work, a relation $R$ from a set $X$ to a set $Y$ is simply a subset of $Y \times X$. Let $\pi_X$ and $\pi_Y$ denote the coordinate projections from $Y \times X$ onto $X$ and $Y$, respectively. Relations are more convenient than maps; however, both will be utilised throughout this work.
	
	\begin{definition}[\cite{MVmodules}, section 3.2]
		Let $(X, \mathcal{E}^X)$ and $(Y, \mathcal{E}^Y)$ be coarse spaces. A relation $R \subset Y \times X$ is said to be controlled if
		$$
			R \circ \mathcal{E}^X \circ R^T \subset \mathcal{E}^Y,
		$$
		that is, for any $E \in \mathcal{E}^X$, the set $R \circ E \circ R^T$ is an entourage. A controlled relation $R$ is called
		\begin{enumerate}
    		\item \emph{densely defined} if $\pi_X(R) \asymp X$;
    		\item \emph{coarsely surjective} if $\pi_Y(R) \asymp Y$;
    		\item a \emph{partial coarse embedding} if $R^T$ is a controlled relation.
		\end{enumerate}
	\end{definition}
	
	For coarse spaces $(X, \mathcal{E}^X)$ and $(Y, \mathcal{E}^Y)$, observe that a partial coarse embedding $R \subset Y \times X$ is necessarily proper, meaning that the preimage of any bounded set in $Y$ is bounded in $X$. Indeed, for a bounded set $B \subset Y$, we have  
	$$  
		R^T[B] \times R^T[B] = R^T \circ (B \times B) \circ R \in \mathcal{E}^Y,  
	$$  
	since $R^T$ is a controlled relation.

	\begin{remark}
		As demonstrated in \cite[Lemma 3.31]{MVmodules}, any densely defined coarse relation defines a coarse map. Conversely, the graph of a coarse map defines a densely defined controlled relation. Thus, we have a correspondence between the language of coarse maps and controlled relations.
		\begin{center}
			\renewcommand{\arraystretch}{1.5} 
			\begin{tabular}{| c | c |} 
    			\hline
    			\textbf{Partial maps} & \textbf{Relations} \\ [2ex]
    			\hline
    			Partial coarse map & Controlled relation \\ [1ex]
    			Coarse map & Densely defined controlled relation \\ [1ex]
    			Cobounded partial coarse map & Coarsely surjective controlled relation \\ [1ex]
    			Expansive coarse map & Densely defined partial coarse embedding \\ [1ex]
    			Coarse equivalence & Coarsely surjective, densely defined partial coarse embedding \\ [1ex]
    			Close coarse maps & Asymptotic controlled relations \\ [1ex]
    			\hline
			\end{tabular}
		\end{center}
		For further details, we refer the reader to \cite{MVmodules, MVrigidity}. A result that will play a significant role in the subsequent developments is the following.
	\end{remark}
	
	\begin{lemma}[\cite{MVmodules}, Corollary 3.42] \label{lemma: included relations => asymptotic}
    	Let $(X, \mathcal{E}^X)$ and $(Y, \mathcal{E}^Y)$ be coarse spaces, and let $R_1$ and $R_2$ be controlled relations from $X$ to $Y$. If $R_1 \prec R_2$, and $\pi_X(R_2) \prec \pi_X(R_1)$, then $R_1 \asymp R_2$.
	\end{lemma}

	The coarse category $\textbf{Coarse}$ is a category whose objects are coarse spaces and for two coarse spaces $(X, \mathcal{E}^X)$ and $(Y, \mathcal{E}^Y)$, the morphisms from $(X, \mathcal{E}^X)$ to $(Y, \mathcal{E}^Y)$ are coarse maps modulo closeness. The morphisms between coarse spaces can be treated as densely defined controlled relations modulo asymptotic equivalence. \newline

	We are now ready to introduce a measurable-like structure on coarse spaces that will substitute the algebra of Borel subsets in Roe's definition of coarse modules \cite{RoeCCaIToCRM}[Definition 4.2]. A partition $\{A_i\}_{i \in I}$ of a coarse space $(X, \mathcal{E}^X)$ is said to be \emph{locally finite} if every bounded subset of $X$ intersects at most finitely many elements of the partition. It is said to be \emph{controlled} if there exists a gauge $E \in \mathcal{E}^X$ such that $A_i$ is $E$-bounded for all $i \in I$.
	
	\begin{definition} \label{definition: discrete space}
		A \emph{locally finite coarse measurable (LFCM) space} is a triple $(X, \mathcal{E}^X, \mathfrak{A}^X)$, where $(X, \mathcal{E}^X)$ is a coarse space, and $\mathfrak{A}^X$ is a Boolean algebra of subsets of $X$ for which there exists a measurable, locally finite, controlled partition $\{A_i\}_{i \in I}$ of $X$ such that $B \in \mathfrak{A}^X$ if and only if $B \cap A_i \in \mathfrak{A}^X$ for all $i \in I$.
	\end{definition}
	
	Throughout the paper, LFCM spaces will be denoted by letters in a special calligraphic font ($\mathcal{X}, \mathcal{Y}, \mathcal{Z}, \ldots$) to simplify notation. The partition of a LFCM space used in the definition will be referred to as a discrete partition, and the gauge that controls its elements will be referred to as a discreteness gauge (denoted by $E^X_{\text{disc}}$). Note that for a fixed discrete partition $\{A_i\}_{i \in I}$ the discreteness gauge $E^X_{\text{disc}}$ may be chosen to be a disjoint union of blocks of the form $A_i \times A_i$, i.e.
	$$
		E^X_{\text{disc}} = \bigsqcup_{i \in I} A_i \times A_i.
	$$
	
	\begin{definition}
		Let $\mathcal{X}$ and $\mathcal{Y}$ be LFCM spaces. A map $f \colon X \to Y$ is said to be a coarse measurable map if $f \colon (X, \mathcal{E}^X) \to (Y, \mathcal{E}^Y)$ is coarse, and $f^{-1}(\mathfrak{A}^Y) \subset \mathfrak{A}^X$.
	\end{definition}
	
	A coarse space is said to be \textit{locally finite} if every bounded subset of it is finite. In the case of a locally finite metric space $(X, d)$, the algebra of Borel subsets is simply the power set $\mathcal{P}(X)$. We equip all locally finite coarse spaces with the power set Boolean algebra (which implies that the discrete partition is a partition of $X$ into singletons, and $E^X_{\text{disc}} = \Delta_X$). This leads to a functor from the full subcategory of \textbf{Coarse} whose objects are locally finite coarse spaces to the category $\textbf{LFCM}$ of LFCM spaces, where the morphisms are $\sim$-equivalence classes of coarse maps.
	
	\begin{proposition} \label{proposition: discritization of LFCM spaces}
		Let $\mathcal{X}$, $\mathcal{Y}$ be LFCM spaces, and let $f \colon (X, \mathcal{E}^X) \to (Y, \mathcal{E}^Y)$ be a coarse map. Then, there exists a coarse measurable map $\tilde{f} \colon \mathcal{X} \to \mathcal{Y}$ that is close to $f$.
	\end{proposition}
	\begin{proof}
		Fix a discrete partition $\{A_i\}_{i \in I}$ of $\mathcal{X}$ and consider a triple $\mathcal{I}_X = (I, \mathcal{E}^I, \mathcal{P}(I))$, where $E \subset I \times I$ belongs to $\mathcal{E}^I$ if and only if
		$$
			\bigsqcup_{(i,j) \in E} A_i \times A_j \in \mathcal{E}^X.
		$$
		It is straightforward to verify that $\mathcal{I}_X$ defines a LFCM space, with its discrete partition given by the partition of $I$ into singletons. There is an evident projection $\pi_X \colon X \to I$ that collapses each $A_i$ to a point $i \in I$. Note that $\pi_X$ is measurable, coarse, expansive, and surjective, thus constituting a coarse equivalence. Let $i_X \colon I \to X$ be a coarse inverse of $\pi_X$. Since the measurable structure on $I$ is provided by the power set $\mathcal{P}(I)$, it follows that $i_X$ is measurable, and consequently, $\pi_X$ is a bimeasurable coarse equivalence. Similarly, select a discrete partition $\{B_j\}_{j \in J}$ to obtain a LFCM space $(J, \mathcal{E}^J, \mathcal{P}(J))$, which is coarsely equivalent to $\mathcal{Y}$ via a bimeasurable map $\pi_Y$. Consider the following diagram:
		$$
			\begin{tikzcd}
    			\mathcal{X} \arrow{d}{\pi_X} \arrow{r}{f} & \mathcal{Y} \\
    			\mathcal{I} \arrow{r}{\tilde{f}_0} & \mathcal{J} \arrow{u}{i_Y},
			\end{tikzcd}
		$$
		where $\tilde{f}_0 = \pi_Y \circ f \circ i_X$. The map $\tilde{f}_0$ is coarse as it is a composition of coarse maps, and it is measurable since the measurable structures on $\mathcal{I}_X$ and $\mathcal{I}_Y$ are simply the power sets. Define $\tilde{f} = i_Y \circ \tilde{f}_0 \circ \pi_X$. It follows that $\tilde{f}$ is a coarse measurable map, and since $i_Y \circ \pi_Y \sim \operatorname{id}_Y$, and $i_X \circ \pi_X \sim \operatorname{id}_X$, one concludes that $f \sim \tilde{f}$.
	\end{proof}
	
	As one may observe from the proof, the spaces under consideration are precisely the locally finite coarse spaces modulo large bounded parts. Despite the seemingly restrictive conditions of Definition \ref{definition: discrete space}, LFCM spaces encompass all proper metric spaces equipped with the Borel algebra of measurable subsets.
	
	\begin{corollary}
		Let $\textbf{LC}$ denote the full subcategory of $\textbf{Coarse}$ whose objects are locally finite coarse spaces. The functor $\mathcal{P} \colon \textbf{LC} \to \textbf{LFCM}$ that sends a locally finite coarse space $(I, \mathcal{E}^I)$ to a LFCM space $(I, \mathcal{E}^I, \mathcal{P}(I))$ is an equivalence of categories.
	\end{corollary}
	
	\subsection{Coarse geometric modules}
	
	Coarse modules were first introduced by Roe \cite{RoeCCaIToCRM} for proper metric spaces as non-degenerate representations of the algebra of compactly supported continuous complex-valued functions. The authors of \cite{MVmodules} presented a topologically independent definition of coarse modules for a coarse space $ (X, \mathcal{E}) $ as a triple $ (\mathfrak{A}, H, \mathbbm{1}_{\bullet}) $, where $ \mathfrak{A} $ is a Boolean algebra of subsets of $ X $, $ H $ is a Hilbert space, and $ \mathbbm{1}_{\bullet} $ denotes a representation of $ \mathfrak{A} $ on $ H $. In light of our decision to fix the Boolean algebra, we will adapt the aforementioned definition accordingly.
	
	\begin{definition}
		Let $ \mathcal{X} $ be a LFCM space. A coarse $ \mathcal{X} $-module $ C $ is a pair $ (H_C, \mathbbm{1}_{\bullet}^C) $, where $ H_C $ is a Hilbert space and $ \mathbbm{1}_{\bullet}^C $ is a representation of $ \mathfrak{A}^X $ on $ H_C $, such that for some gauge $ E_{\text{ndeg}}^X $, the set
		$$
    		H_C^0 = \operatorname{span} \{ \mathbbm{1}_A^C H_C \mid A \in \mathfrak{A}^X \text{ is } E^X_{\text{ndeg}} \text{-bounded} \}
		$$
		is dense in $ H_C $. The gauge $E^X_{\text{ndeg}}$ will be referred to as a non-degeneracy gauge.
	\end{definition}
	
	As noted in \cite{MVmodules}, for a proper metric space $ (X, d) $, the object defined above coincides with a non-degenerate representation of $ C_0(X) $ when $ \mathfrak{A}^X$ is taken to be the algebra of Borel measurable sets.
	
	\begin{remark}
		According to the conditions imposed on the algebra of measurable sets of a LFCM space $ \mathcal{X} $, it can be verified that any coarse geometric module as defined here is a discrete coarse geometric module in the sense of \cite{MVmodules}. The discreteness gauge may be taken as the non-degeneracy gauge. It is important to observe that the local finiteness of the discrete partition implies that
		$$
			H_C^0 = \operatorname{span} \{\mathbbm{1}_A^C H \mid A \in \mathfrak{A}^X \text{ is bounded}\}.
		$$
	\end{remark}
	
	\begin{definition}[\cite{MVmodules}, Definition 4.10]
		Let $ \mathcal{X} $ be a LFCM space. A coarse $\mathcal{X}$-module $C$ is \emph{faithful} if the set
		$$
			\bigcup \{ A \in \mathfrak{A}^X \mid A \text{ is } E^X_{\text{disc}} \text{-bounded and } \mathbbm{1}_A^C \neq 0 \}
		$$
		is coarsely dense in $X$.
	\end{definition}
	
	Our primary focus is on bounded operators acting between the underlying Hilbert spaces of coarse modules, which, with a slight abuse of notation, will be denoted by $B(C_X, C_Y)$. Throughout this paper, coarse modules will be denoted by plain capital letters, such as $C, D, \ldots$. If it is necessary to emphasise that $C$ is an $\mathcal{X}$-module, we shall include the respective subscript $C_X$.
	
	\begin{definition}[\cite{MVmodules}, Definition 5.6]
		Let $\mathcal{X}$, $\mathcal{Y}$ be LFCM spaces, $C_X$, $C_Y$ be coarse modules, and $t \colon H_{C_X} \to H_{C_Y}$ be a bounded operator. The support of $t$ is the following subset of $Y \times X$:
		$$
			\operatorname{Supp}(t) = \bigcup \{B \times A \mid A \in \mathfrak{A}^X \text{ is } E_{\text{ndeg}}^X \text{-bounded, } B \in \mathfrak{A}^Y \text{ is } E_{\text{ndeg}}^Y \text{-bounded}, \mathbbm{1}_B t \mathbbm{1}_A \neq 0\}.
		$$
	\end{definition}
	
	In light of \cite[Proposition 5.7]{MVmodules}, the asymptotic equivalence class of the support defined above coincides with the concept of support as introduced in \cite{MVmodules}. In fact, up to asymptotic equivalence, the defined object is independent of the choice of the non-degeneracy gauge, thereby permitting the selection of the discreteness gauge. To simplify the notation, from this point onwards, the non-degeneracy gauge of a coarse module will always be taken to be the discreteness gauge of the LFCM space.
	
	\begin{lemma}[\cite{MVmodules}, Lemma 5.3] \label{lemma: properties of support}
		Let $\mathcal{X}$, $\mathcal{Y}$, and $\mathcal{Z}$ be LFCM spaces, with $C_X$, $C_Y$, and $C_Z$ denoting coarse modules. Consider bounded operators $t_1, t_2 \colon H_{C_X} \to H_{C_Y}$ and $s \colon H_{C_Y} \to H_{C_Z}$. Then the following hold:
		\begin{enumerate}
			\item (Support and adjoints) $\operatorname{Supp}(s^*) = \operatorname{Supp}(s)^T$;
		
			\item (Support and sums) $\operatorname{Supp}(t_1 + t_2) \subset \operatorname{Supp}(t_1) \cup \operatorname{Supp}(t_2)$;
		
			\item (Support and compositions) $\operatorname{Supp}(s t_1) \subset \operatorname{Supp}(s) \circ E^Y_{\text{disc}} \circ \operatorname{Supp}(t_1)$.
		\end{enumerate}
	\end{lemma}
	\begin{proof}
		The proof of the first two statements is identical to that in \cite{MVmodules}. For the third assertion, let $\{B_i\}_{i \in I}$ be a discrete partition of $Y$. Since the sum in the Strong Operator Topology (SOT) of $\mathbbm{1}_{B_i}$ over $I$ converges to the identity, it follows that for any product of measurable sets $C \times A$ from the union in the definition of the support of $s t_1$, there exists $i \in I$ such that
		$$
			\mathbbm{1}_C s \mathbbm{1}_{B_i} \neq 0, \quad \text{and} \quad \mathbbm{1}_{B_i} t_1 \mathbbm{1}_A \neq 0.
		$$
		Therefore, it follows that $C \times A \subset \operatorname{Supp}(s) \circ E^Y_{\text{disc}} \circ \operatorname{Supp}(t_1)$.
	\end{proof}
	
	Note that the third statement of the preceding lemma is stronger than the corresponding one in \cite{MVmodules}. This enhancement arises from the fact that the modules considered in our work are discrete, as defined in \cite[Definition 4.18]{MVmodules}. The classes of operators of particular interest in this paper are summarised in the following definition.
	
	\begin{definition}[\cite{MVrigidity}]
		Let $\mathcal{X}$, $\mathcal{Y}$ be LFCM spaces, and $C_X^1$, $C_X^2$, $C_Y$ be coarse modules. Let $t \colon H_{C_X^1} \to H_{C_X^2}$ and $g \colon H_{C_X^1} \to H_{C_Y}$ be bounded operators.
		\begin{enumerate}
    		\item $g$ is \emph{controlled} if the support of $g$ is a controlled relation;
    		\item \begin{enumerate}
        		\item $t$ is said to have \emph{controlled propagation} if its support is an entourage. For an entourage $E$ in $\mathcal{X}$, the operator $t$ is said to have \emph{$E$-controlled propagation} if its support is contained within $E$;
        		\item For an entourage $E$ in $\mathcal{X}$ and $\varepsilon > 0$, the operator $t$ is said to be \emph{$(\varepsilon, E)$-approximable} if there exists a bounded operator $s$ of $E$-controlled propagation such that $\|t - s\| \le \varepsilon$;
        		\item $t$ is said to be \emph{approximable} if for every $\varepsilon > 0$, there exists an entourage $E \in \mathcal{E}_X$ such that $t$ is $(\varepsilon, E)$-approximable.
    		\end{enumerate}
		\end{enumerate}
	\end{definition}

	For a detailed exposition of the properties of the aforementioned classes of operators, we refer to \cite{MVmodules, MVrigidity}. When we wish to emphasise that $t \colon H_{C_X^1} \to H_{C_X^2}$ is approximable, we will write $t \colon C_X^1 \to C_X^2$. The motivation for this notation will become apparent shortly. We conclude this section with several propositions that are required for the subsequent results.

	\begin{proposition} \label{proposition: approximable operators form a Banach space}
		Let $\mathcal{X}$, $\mathcal{Y}$ be LFCM spaces, and $C_X^1$, $C_X^2$, $C_Y$ be coarse modules. The set of approximable operators from $C_X^1$ to $C_X^2$ forms a Banach space. The set of approximable operators from $C_X^1$ to itself forms a $C^*$-algebra.
	\end{proposition}
	
	Controlled operators will serve as $*$-homomorphisms between $C^*$-algebras of approximable operators as suggested by the following proposition.
		
	\begin{proposition}[\cite{MVmodules}, Proposition 7.1] \label{proposition: properties of controlled operators}
		Let $\mathcal{X}$ and $\mathcal{Y}$ be LFCM spaces. Let $s_1, s_2 \colon H_{C_X} \to H_{C_Y}$ with asymptotically equivalent supports. The following are equivalent:
		\begin{enumerate}
    		\item The operators $s_1$ and $s_2$ are controlled;
    		\item For every $E \in \mathcal{E}^X$, there exists $F \in \mathcal{E}^Y$ such that for every $E$-controlled propagation operator $t \colon C_X \to C_X$, the operator $s_1 t s_2^*$ has $F$-controlled propagation.
		\end{enumerate}
	\end{proposition}
	
	\begin{proposition}[\cite{MVmodules}, Lemma 7.17] \label{proposition: approximable, coarse-like unitary => controlled}
		Let $\mathcal{X}$ be a LFCM space, and let $C_1$, $C_2$ be coarse $\mathcal{X}$-modules. Any controlled, approximable unitary $u \colon H_{C_1} \to H_{C_2}$ has controlled propagation.
	\end{proposition}
	
	\subsection{Approximable category}
	
	In this subsection, we introduce the primary objects of study, specifically the approximable category. Related constructs have been explored in \cite{BunkeEngel_HomThwBornCoarseSp} within the framework of addressing the Baum--Connes conjecture.
	
	\begin{definition}
		Let $\mathcal{X}$ be a LFCM space. For an infinite cardinal $\kappa$, define the category $\mathcal{A}_{\kappa}(\mathcal{X})$, whose objects are coarse $\mathcal{X}$-modules of rank at most $\kappa$. Morphisms between two coarse $\mathcal{X}$-modules $C_0$ and $C_1$ are approximable operators from $C_0$ to $C_1$. Additionally, define the category $\mathcal{A}_{\infty}(\mathcal{X})$ with objects consisting of coarse $\mathcal{X}$-modules of arbitrary rank and morphisms given by approximable operators. Hereafter, $\mathcal{A}_{\infty}(\mathcal{X})$ (respectively, $\mathcal{A}_{\kappa}(\mathcal{X})$) will be referred to as the \emph{approximable category} (respectively, the \emph{approximable category of rank $\kappa$}) of $\mathcal{X}$.
	\end{definition}
	
	Most of the theory developed in this work applies equally to $\mathcal{A}_{\infty}(\mathcal{X})$ and $\mathcal{A}_{\kappa}(\mathcal{X})$. Accordingly, we introduce the notation $\mathcal{A}(\mathcal{X})$ to denote either $\mathcal{A}_{\kappa}(\mathcal{X})$ or $\mathcal{A}_{\infty}(\mathcal{X})$. Furthermore, we adopt the convention that if an expression involves $\mathcal{A}(\mathcal{X})$ multiple times, then all instances refer either to $\mathcal{A}_{\kappa}(\mathcal{X})$ for a fixed $\kappa$, or to $\mathcal{A}_{\infty}(\mathcal{X})$ exclusively.
	
	Lemma \ref{lemma: properties of support} establishes that $\mathcal{A}(\mathcal{X})$ is a well-defined category. Moreover, a straightforward verification, combined with Proposition \ref{proposition: approximable operators form a Banach space}, demonstrates that $\mathcal{A}(\mathcal{X})$ forms a $C^*$-category. For a detailed treatment of $C^*$-categories, we refer the reader to \cite{BunkeEngel_AdditiveCstarCat, Mitchener_CstarCats}. The approximable category of rank $\kappa$ can be regarded as a small $C^*$-category.

	Suppose we are given a coarse measurable map $f \colon \mathcal{X} \to \mathcal{Y}$ between LFCM spaces. For an $\mathcal{X}$-module $C$, define a representation $\mathbbm{1}_{\bullet}^{f_*C}$ of $\mathfrak{A}^Y$ on $H_C$ by 
	$$
		\mathbbm{1}_{A}^{f_*C} = \mathbbm{1}_{f^{-1}(A)}^C.
	$$
	Since $f$ is measurable, this yields a well-defined representation. Analogous constructions have been studied in \cite[page 41, 4.9]{RoeCCaIToCRM} and \cite[Definition 8.9]{BunkeEngel_HomThwBornCoarseSp}.
	
	\begin{lemma} \label{lemma: Pushforward of a module}
    Let $\mathcal{X}$ and $\mathcal{Y}$ be LFCM spaces, $f \colon \mathcal{X} \to \mathcal{Y}$ be a coarse measurable map, and $C$ be an $\mathcal{X}$-module. The representation $\mathbbm{1}^{f_*(C)}_{\bullet}$ of $\mathfrak{A}^Y$ on $H_C$ is non-degenerate. Consequently, it defines a coarse $\mathcal{Y}$-module.
	\end{lemma}
	\begin{proof}
    	Let $\{A_i\}_{i \in I}$ and $\{B_j\}_{j \in J}$ be discrete partitions of $\mathcal{X}$ and $\mathcal{Y}$ respectively. Observe that the set $\{f^{-1}(B_j)\}_{j \in J} \subset \mathfrak{A}^X$ forms a locally finite cover of $\mathcal{X}$. Consequently, for every $i \in I$, there exists a finite subset $J_i \subset J$ such that 
		$$
			A_i \subset \bigcup_{j \in J_i} f^{-1}(B_j), \qquad \qquad \mathbbm{1}_{A_i}^C H_C \subset \sum_{j \in J_i} \mathbbm{1}_{B_j}^{f_*C} H_C.
		$$
		The right-hand side of the above expression is merely the application of $\mathbbm{1}_{\bullet}^C$ to the left-hand side. The conclusion follows immediately.
	\end{proof}

	Denote the $\mathcal{Y}$-module $(\mathbbm{1}^{f_*(C)}_{\bullet}, H_C)$ by $f_*C$. We refer to $f_*(C)$ as the \emph{pushforward} of $C$ along $f$. Observe that the identity map on $H_C$, regarded as a bounded operator between $C$ and $f_*C$, is supported on $f$. In particular, the unitary
		$$
			U_f(C) \colon C \to f_*C, \quad U_f(C) = \operatorname{id}_{H_C}
		$$
	is a controlled unitary. Therefore, $\operatorname{Ad}_{U_f(C)}$ maps controlled propagation operators on $C$ to controlled propagation operators on $f_*C$, by Proposition \ref{proposition: properties of controlled operators}. Observe that the assumption of $\mathcal{Y}$ being LFCM is essential for Lemma \ref{lemma: Pushforward of a module} to hold; this is the primary justification for working with LFCM spaces.

	\begin{definition}
    	Let $\mathcal{X}$ and $\mathcal{Y}$ be LFCM spaces and $f \colon \mathcal{X} \to \mathcal{Y}$ be a coarse measurable map. Define a functor $f_* \colon \mathcal{A}(\mathcal{X}) \to \mathcal{A}(\mathcal{Y})$, which maps a coarse $\mathcal{X}$-module $C$ to its pushforward $f_*C$ along $f$, and a morphism $t \colon C_0 \to C_1$ to the conjugation $f_*t = U_f(C_1) t U_f(C_0)^*$.
	\end{definition}

	As $\operatorname{Ad}_{U_f(C)}$ preserves controlled propagation operators, $f_*$ is a well-defined functor. Note that $f_*(t^*) = f_*(t)^*$, so $f_*$ is a $*$-functor. In particular, the correspondence $\mathcal{X} \to \mathcal{A}(\mathcal{X})$ gives rise to a functor from the category of LFCM spaces and coarse maps to the category of all $C^*$-categories. Suppose we are given two coarse measurable maps $f, g \colon \mathcal{X} \to \mathcal{Y}$ that are close to each other. For a coarse $\mathcal{X}$-module $C$, consider the unitaries
	$$
		\eta_C \colon f_*(C) \to g_*(C), \quad \eta_C = U_g(C) U_f(C)^*.
	$$
	Note that the following diagram commutes for every pair of coarse $\mathcal{X}$-modules $C_0$ and $C_1$, and any approximable operator $t \colon C_0 \to C_1$:
	$$
		\begin{tikzcd}
    		f_*(C_0) \arrow{d}{\eta_{C_0}} \arrow{r}{f_*t} & f_*(C_1) \arrow{d}{\eta_{C_1}} \\
    		g_*(C_0) \arrow{r}{g_*t} & g_*(C_1)
		\end{tikzcd}
	$$
	It is important to observe that the support of $\eta_C$ is contained within $(f \times g)(E_{\text{disc}}^X)$, and since $f$ is close to $g$, the unitaries $\eta_C$ have controlled propagation.

	\begin{corollary}
    	For two coarse measurable maps $f, g \colon \mathcal{X} \to \mathcal{Y}$ that are close to each other, the induced functors $f_*, g_*$ are naturally isomorphic. In particular, the correspondence $\mathcal{X} \mapsto \mathcal{A}(\mathcal{X})$ extends to a functor between the category \textbf{LFCM} and the category of all $C^*$-categories, with classes of natural equivalence of functors as morphisms.
	\end{corollary}

	The approximable category possesses more structure than that of a mere $C^*$-category. In particular, it is finitely additive. We refer to \cite[Chapter 8, section 2]{maclane} for a comprehensive treatment of additive categories.
		
	\begin{definition}
		Let $C_0$, $C_1$ be coarse $\mathcal{X}$-modules. Consider the coarse $\mathcal{X}$-module $C_0 \oplus C_1$, whose underlying Hilbert space is $H_{C_0} \oplus H_{C_1}$, with the representation given by
		$$
			A \mapsto \mathbbm{1}_A^{C_0} \oplus \mathbbm{1}_A^{C_1}.
		$$
	\end{definition}
		
	Verifying that the representation defined above satisfies the nondegeneracy condition and thus defines a coarse $\mathcal{X}$-module is straightforward. The canonical inclusions $i_{k}\colon C_k \to C_0 \oplus C_1$, for $k = 0, 1$, are controlled propagation operators.
		
	\begin{lemma} \label{lemma: Finite coproducts}
    	Let $\mathcal{X}$ be a LFCM space. The direct sum of coarse $\mathcal{X}$-modules is the coproduct in the approximable category $\mathcal{A}(\mathcal{X})$.
	\end{lemma}
	\begin{proof}
    	Since the underlying Hilbert space of $C_0 \oplus C_1$ is $H_{C_0} \oplus H_{C_1}$, for any coarse $\mathcal{X}$-module $C$ and approximable maps $\eta_k \colon C_k \to C$, with $k = 0, 1$, there exists a unique bounded operator $\eta \colon C_0 \oplus C_1 \to C$ given by
		$$
			\eta = \eta_1 i_1^* + \eta_2 i_2^*.
		$$
		Since sums, compositions, and adjoints of approximable operators are again approximable, the operator $\eta$ is a morphism in $\mathcal{A}(\mathcal{X})$, which ensures that the coproduct diagram commutes.
	\end{proof}
	
	It follows from \cite[page 228]{maclane} that $\mathcal{A}(\mathcal{X})$ is additive. Indeed, it is enriched in Banach spaces and, thus, in Abelian groups, making it a preadditive category. Since $\mathcal{A}(\mathcal{X})$ has all finite coproducts, it also has all finite products, which are biproducts. Consequently, $\mathcal{A}(\mathcal{X})$ is additive. Furthermore, $*$-functors are $\mathbb{C}$-linear by definition, and thus they are $\textbf{Ab}$-enriched, which implies they are additive.
		
	\begin{corollary}
		For a LFCM space $\mathcal{X}$, the approximable category $\mathcal{A}(\mathcal{X})$ is an additive $C^*$-category. The functors induced by coarse measurable maps are additive $*$-functors.
	\end{corollary}
	
	For a coarse $\mathcal{X}$-module $C$ the $C^*$-algebra $\operatorname{End}_{\mathcal{A}(\mathcal{X})}(C)$ of approximable operators on $H_C$ is called the \emph{algebra of approximable operators on $C$}. Our definition of the aforementioned $C^*$-algebra coincides with the one given in \cite{MVmodules, MVrigidity}.
	
	\subsection{Rigidity of $*$-isomorphisms for countably generated spaces}
	
	The rigidity problem concerns whether, given a $*$-isomorphism between the algebras of approximable operators on faithful modules, one can deduce a coarse equivalence between the underlying coarse spaces. This question was completely resolved for countably generated coarse spaces in \cite{BFBFKVW_rigidity} and \cite{MVrigidity}. In this section, we recall the strategy employed in \cite{MVrigidity}. Although the results apply to a more general setting than LFCM spaces, we focus on the LFCM case for simplicity.
	
	\begin{theorem}[\cite{MVrigidity}, Proposition 7.2] \label{theorem: spatial implementation}
		Let $\mathcal{X}$ and $\mathcal{Y}$ be LFCM spaces, and let $C_X$ and $C_Y$ be faithful coarse modules. Suppose 
		$$
			\phi \colon \operatorname{End}_{\mathcal{A}(\mathcal{X})}(C_X) \to \operatorname{End}_{\mathcal{A}(\mathcal{Y})}(C_Y)
		$$
		is a $*$-isomorphism. Then, there exists a unitary operator $U_{\phi} \colon H_{C_X} \to H_{C_Y}$ such that $\phi = \operatorname{Ad}_{U_{\phi}}$.
	\end{theorem}
	
	The unitary obtained in Theorem \ref{theorem: spatial implementation} is not unique; it is unique up to central approximable unitaries. Specifically, for any two unitaries $U_{\phi}^1, U_{\phi}^2 \colon H_{C_X} \to H_{C_Y}$ that implement $\phi$, there exists a central unitary $V \in \operatorname{End}_{\mathcal{A}(\mathcal{X})}(C_X)$ such that $U_{\phi}^1 = U_{\phi}^2 V$. 
	\begin{remark}
		Let $\mathcal{X}$ be an LFCM space, and let $\{X_k\}_{k \in J}$ denote the decomposition of $X$ into its coarsely connected components. Each such component is measurable, as $X_k = E^X_{\text{disc}}[X_k]$ for every $k \in J$. Let $C$ be a coarse $\mathcal{X}$-module. The central unitaries of the algebra of approximable operators $\operatorname{End}_{\mathcal{A}(\mathcal{X})}(C)$ on $C$ are precisely of the form  
		$$
			U = \operatorname{SOT-} \sum_{k \in J} \alpha_k \mathbbm{1}_{X_k},
		$$  
		for some collection of complex numbers $\{\alpha_k\}_{k \in J}$ of modulus one. Indeed, by \cite[Corollary 6.18]{MVmodules}, the algebra $\operatorname{End}_{\mathcal{A}(\mathcal{X})}(C)$ contains the algebra of compact operators on $\mathbbm{1}_{X_k} H_C$ for each $k \in J$. Consequently,  
		$$
			\mathbbm{1}_{X_k} U = U \mathbbm{1}_{X_k} = \alpha_k \mathbbm{1}_{X_k}.
		$$  
		By summing in the strong operator topology, the result follows.
	\end{remark}
	The desired coarse equivalence is extracted from the unitaries $U_{\phi}$.
	
	\begin{definition}[\cite{MVrigidity}, Definition 6.3]
		Let $\mathcal{X}$ and $\mathcal{Y}$ be LFCM spaces, and let $T \colon H_{C_X} \to H_{C_Y}$ be a bounded operator. For $\delta > 0$, $F \in \mathcal{E}^Y$, and $E \in \mathcal{E}^X$, the approximate relation of $T$ is a subset of $Y \times X$ defined by
		$$
			f^T_{\delta, F, E} = \bigcup \{B \times A \mid B \in \mathfrak{A}^Y \text{ is } F \text{-bounded, } A \in \mathfrak{A}^X \text{ is } E \text{-bounded, } \|\mathbbm{1}_B^{C_Y} T \mathbbm{1}_A^{C_X}\| > \delta\}.
		$$
	\end{definition}
	
	\begin{remark} \label{remark:: folded approx relations when folded param}
		It is trivial to see that if $\delta_2 < \delta_1$, $F_1 \subset F_2 \in \mathcal{E}^Y$, and $E_1 \subset E_2 \in \mathcal{E}^X$, then the approximate relations satisfy $f^T_{\delta_1, F_1, E_1} \subset f^T_{\delta_2, F_2, E_2}$.
	\end{remark}
	
	For arbitrary bounded operators, approximate relations yield no information about the underlying LFCM spaces. However, a specific class of operators renders approximate relations as partial coarse maps.
	
	\begin{definition}[\cite{BVFcoronas}, Definition 3.2]
		An operator $T \colon H_{C_X} \to H_{C_Y}$ is said to be \emph{coarse-like} if, for every $\varepsilon > 0$ and $E \in \mathcal{E}^X$, there exists $F \in \mathcal{E}^Y$ such that the map $\operatorname{Ad}_T \colon B(H_{C_X}) \to B(H_{C_Y})$ sends $E$-controlled propagation contractions to $(\varepsilon, F)$-approximable operators.
	\end{definition}
	
	The authors of \cite{MVrigidity} use the term \emph{weakly-approximately controlled} to refer to coarse-like operators. Whether a unitary implementing a $*$-isomorphism between algebras of approximable operators on faithful modules over LFCM spaces is coarse-like remains an open problem. The following theorem was first established in \cite{BFVemb2019} for uniformly locally finite metric spaces.
	
	\begin{theorem}[\cite{MVrigidity}, Theorem 5.5] \label{theorem: Iso => unitary is coarse-like}
		Let $\mathcal{X}$ and $\mathcal{Y}$ be countably generated LFCM spaces, $C_X$ and $C_Y$ be faithful coarse modules, and let $U \colon H_{C_X} \to H_{C_Y}$ be a unitary operator such that $\operatorname{Ad}_{U}$ restricts to a $*$-isomorphism
		$$
			\operatorname{Ad}_{U} \colon \operatorname{End}_{\mathcal{A}(\mathcal{X})}(C_X) \to \operatorname{End}_{\mathcal{A}(\mathcal{Y})}(C_Y).
		$$
		Then $U$ is coarse-like.
	\end{theorem}
	
	The assumption that the spaces are countably generated is essential for the proof strategy, as the core argument relies on the application of the Baire category theorem.
	
	\begin{theorem}[\cite{MVrigidity}, Corollary 6.20] \label{theorem: rigidity of $*$-isomorphisms}
		Let $\mathcal{X}$ and $\mathcal{Y}$ be countably generated LFCM spaces, and let $C_X$ and $C_Y$ be faithful coarse modules. Suppose $U \colon H_{C_X} \to H_{C_Y}$ is a unitary operator such that $\operatorname{Ad}_{U}$ restricts to a $*$-isomorphism:
		$$
			\operatorname{Ad}_{U} \colon \operatorname{End}_{\mathcal{A}(\mathcal{X})}(C_X) \to \operatorname{End}_{\mathcal{A}(\mathcal{Y})}(C_Y).
		$$
		Then, for any $\delta \in (0,1)$, there exist sufficiently large $F \in \mathcal{E}^Y$ and $E \in \mathcal{E}^X$ such that $f^U_{\delta, F, E}$ is a coarse equivalence with inverse $f^{U^*}_{\delta, E, F}$.

	\end{theorem}
	
	Compare the above results with their analogues in \cite{SWrigidity, BF_rigidity, BV_GelfandTypeDuality, BFBFKVW_rigidity}.
	
	\section{Domains of coarse modules}
	
	Some modules more effectively capture the geometry of the underlying LFCM spaces than others. The following two examples demonstrate this phenomenon.
	
	\begin{example} \label{example: uniform module}
		Let $\mathcal{X}$ be an LFCM space with a discrete partition $\{A_i\}_{i \in I}$. Suppose that a subset $B \subset X$ is measurable if and only if it can be expressed as a disjoint union of elements of this partition. For a measurable subset $A \subset X$, define  
		$$ 
			I_A = \{i \in I \mid A_i \subset A\}.
		$$ 
		Consider the representation of $\mathfrak{A}^X$ on $\ell^2(I)$ given by  
		$$  
			\mathfrak{A}^X \ni A \mapsto \mathbbm{1}_A^u = \text{SOT-}\sum_{i \in I_A} \mathbbm{1}_i \in \operatorname{Proj}(\mathcal{B}(\ell^2(I))),  
		$$  
		where $\mathbbm{1}_i$ denotes the rank-one projection onto the span of $\delta_i \in \ell^2(I)$. This representation is non-degenerate, as the linear span of $\{\delta_i \mid i \in I\}$ is dense in $\ell^2(I)$. When the discrete partition consists of singletons, the resulting module is commonly referred to as \emph{the uniform module}. For a general LFCM space $\mathcal{X} = (X, \mathcal{E}^X, \mathfrak{A}^X)$, fix a discrete partition $\{A_i\}_{i \in I}$. Define an LFCM space $\mathcal{I} = (I, \mathcal{E}^I, \mathcal{P}(I))$, as in the proof of Proposition \ref{proposition: discritization of LFCM spaces}. The inclusion map $i \colon \mathcal{I} \to \mathcal{X}$ is a bimeasurable coarse equivalence, allowing one to consider the pushforward $i_*C_u$ of the uniform module $C_u$ of $\mathcal{I}$ to $\mathcal{X}$. Note that the pushforward representation satisfies  
		$$
			\mathbbm{1}^{i_*C_u}_{A_j} = \mathbbm{1}^{C_u}_{i^{-1}(A_j)} = \mathbbm{1}^{C_u}_{\{j\}} \neq 0,
		$$
		for all $j \in I$.
	\end{example}

	\begin{example} \label{example: bounded module}
		Let $\mathcal{X}$ be a LFCM space. Fix a point $x \in X$, and consider a representation of $\mathfrak{A}^X$ on $\mathbb{C}$ defined by
		$$
			\mathfrak{A}^X \ni A \mapsto \mathbbm{1}^{x}_A = \begin{cases}
					1, & \text{if } x \in A; \\
					0, & \text{otherwise.}
				\end{cases}
		$$
		The non-degeneracy condition is satisfied trivially. A representation of a similar nature may be defined on a Hilbert space of any dimension. This module will be referred to as a bounded module.
	\end{example}

	These two modules exhibit fundamentally different characteristics. In many instances, the algebra of approximable operators on the uniform module fully encodes the coarse geometry of $\mathcal{X}$ (see \cite{BV_GelfandTypeDuality, BBFVW_BijCoaExpanders} for the case of uniformly locally finite spaces). By contrast, the algebra of approximable operators on a bounded module is always $\mathbb{C}$ and contains no information about the underlying coarse space. It is desirable to construct invariants that distinguish these cases.
	
	\begin{definition}\label{definition:: domains}
		Let $\mathcal{X}$ be a LFCM space, and $C$ be a coarse $\mathcal{X}$-module. For a cardinal $\kappa$, the $\kappa$-domain of $C$ is defined as the subset of $\mathcal{X}$ given by
		$$
			\operatorname{dom}_{\kappa}(C) = \bigcup \{A \in \mathfrak{A}^X \mid A \text{ is } E_{\text{disc}}^X\text{-bounded and } \text{rk}(\mathbbm{1}^C_A) \geq \kappa\}.
		$$
		If $\kappa = 1$, $\operatorname{dom}_{\kappa}(C)$ is called the faithfulness domain of $C$, while for $\kappa = \aleph_0$, it is called the ampleness domain of $C$. A coarse module $C$ is said to be faithful (respectively, ample) if its faithfulness (respectively, ampleness) domain is asymptotic to $\mathcal{X}$.
	\end{definition}
	Observe that for a finite cardinal $\kappa \neq 1$, the asymptotic equivalence class of the $\kappa$-domain depends on the choice of discreteness gauge. To illustrate this, consider a space $\mathcal{X}$ which is a connected pair of points:
	$$
		\mathcal{X} = (\{x,y\}, \mathcal{P}(\{x,y\} \times \{x,y\}), \mathcal{P}(\{x,y\})).
	$$
	Let $C$ be a coarse $\mathcal{X}$-module, whose underlying Hilbert space is $\mathbb{C}^2$, and the representation is given by mapping each point to a one-dimensional projection. Consider two distinct choices of discreteness gauges for $C$:
	$$
		E_1 = \Delta_X, \qquad E_2 = X \times X.
	$$
	With respect to $E_1$, the $2$-domain of $C$ is empty, whereas for $E_2$, the $2$-domain of $C$ coincides with the entire set $X$. This phenomenon does not arise for infinite cardinals $\kappa$ or for $\kappa = 1$, as demonstrated in the following lemma.

	\begin{lemma}
		Let $\mathcal{X}$ be a LFCM space, and $C$ be a coarse $\mathcal{X}$-module. For an infinite cardinal $\kappa$, or $\kappa = 1$, the asymptotic equivalence class of $\kappa$-domain of $C$ is independent of the choice of the discreteness gauge.
	\end{lemma}
	\begin{proof}
		Suppose given a discrete partition $\{A_i\}_{i \in I}$, and a discreteness gauge $E^X_{\text{disc}}$. Consider an alternative discreteness gauge  
		$$
			F^X_{\text{disc}} = \bigsqcup_{i \in I} A_i \times A_i,
		$$
		Let $\operatorname{dom}_{\kappa}^E(C)$ and $\operatorname{dom}_{\kappa}^F(C)$ denote the $\kappa$-domains of $C$ with respect to the initial discreteness gauge $E^X_{\text{disc}}$ and the alternative gauge $F^X_{\text{disc}}$, respectively. Since $F^X_{\text{disc}}$ is contained within $E^X_{\text{disc}}$, every $F^X_{\text{disc}}$-bounded set is necessarily $E^X_{\text{disc}}$-bounded, implying that
		$$
			\operatorname{dom}_{\kappa}^F(C) \subset \operatorname{dom}_{\kappa}^E(C).
		$$
		Conversely, every $E^X_{\text{disc}}$-bounded set $A$ is contained within a finite union $\bigcup_{k=1}^n A_{i_k}$ of elements of the discrete partition, where $A \cap A_{i_k} \neq \emptyset$ for all $1 \leq k \leq n$. Consequently, for an $E^X_{\text{disc}}$-bounded set $A$ satisfying $\operatorname{rk}(\mathbbm{1}_A) \geq \kappa$, we obtain
		$$
			\kappa \leq \operatorname{rk}(\mathbbm{1}_A) \leq \operatorname{rk} \left(\sum_{k=1}^n \mathbbm{1}_{A_{i_k}} \right).
		$$
		Since $\kappa$ is either an infinite cardinal or equal to $1$, it follows that for some $1 \leq k \leq n$, we have $\operatorname{rk}(\mathbbm{1}_{A_{i_k}}) \geq \kappa$. Furthermore, as $A$ is $E^X_{\text{disc}}$-bounded and intersects $A_{i_k}$, it follows that $E^X_{\text{disc}}[A_{i_k}] \supset A$. Thus, we conclude that  
		$$
			\operatorname{dom}_{\kappa}^E \subset E^X_{\text{disc}}[\operatorname{dom}_{\kappa}^F].
		$$
		In particular, the $\kappa$-domains corresponding to the two discreteness gauges are asymptotically equivalent.
	\end{proof}
	
	From now on the term $\kappa$-domain will apply only to infinite cardinals $\kappa$, or $\kappa = 1$. Henceforth, we shall work with measurable representatives of $\kappa$-domains. Observe that if the discreteness gauge is given by  
	$$
		E^X_{\text{disc}} = \bigsqcup_{i \in I} A_i \times A_i,
	$$
	for a discrete partition $\{A_i\}_{i \in I}$, then all $\kappa$-domains are measurable. Indeed, if $A \subset X$ is $E^X_{\text{disc}}$-bounded, then it is contained within some partition element $A_i$, which implies
	$$
		\operatorname{rk}(\mathbbm{1}_A) \leq \operatorname{rk}(\mathbbm{1}_{A_i}).
	$$
	Thus, any $\kappa$-domain is a disjoint union of elements of the discrete partition and is consequently measurable. Following this observation, when referring to $\kappa$-domains, we shall always consider them with respect to the discreteness gauge $E^X_{\text{disc}}$ introduced above.
	
	Faithful and ample modules for proper metric spaces were introduced in \cite{RoeCCaIToCRM}. An equivalent topologically independent definition was provided in \cite{MVmodules}. Note that our definition of faithful and ample modules coincides with that in \cite{MVmodules}.
	
	\begin{definition}
		Let $\mathcal{X}$ be a LFCM space, and let $C_0$, $C_1$ be coarse $\mathcal{X}$-modules. A bounded operator $t \colon H_{C_0} \to H_{C_1}$ is said to be \emph{coarsely full} if $\pi_1(\operatorname{Supp}(t)) \asymp \operatorname{dom}_1(C_1)$.
	\end{definition} 
	
	Coarsely full operators are particularly useful when comparing the domains of coarse modules. Indeed, given two coarse modules $C_0$ and $C_1$ and a coarsely full operator $t \colon H_{C_0} \to H_{C_1}$, we have
	\begin{equation*}
		\begin{split}
			\operatorname{dom}_1(C_1) & \asymp \pi_1(\operatorname{Supp}(t)) = \operatorname{Supp}(t)[\operatorname{dom}_1(C_0)].
		\end{split}
	\end{equation*}
	In particular, if there exists a coarsely full, controlled propagation operator $t \colon C_0 \to C_1$ between coarse $\mathcal{X}$-modules $C_0$ and $C_1$, then the faithfulness domain of $C_1$ is subordinated to that of $C_0$. The following lemma presents examples of coarsely full operators.
	 
	\begin{lemma} \label{lemma: Invertible operator is fully supported}
		Let $\mathcal{X}$ be a LFCM space and $C_0, C_1$ be coarse $\mathcal{X}$-modules. Any bounded invertible operator $g \colon H_{C_0} \to H_{C_1}$ is coarsely full.
	\end{lemma}
	\begin{proof}
		Let $A \subset \operatorname{dom}_1(C_1)$ be a measurable and $E^X_{\text{disc}}$-bounded subset of $X$. Suppose that 
		$$
			A \subset X \setminus \pi_1(\operatorname{Supp}(g)).
		$$
		Then, for every $E^X_{\text{disc}}$-bounded set $B \in \mathfrak{A}^X$, we have $\mathbbm{1}^{C_1}_A g \mathbbm{1}^{C_0}_B = 0$. By summing over all such $B$ in the strong operator topology, it follows that $\mathbbm{1}^{C_1}_A g = 0$. Since $g$ is invertible, this implies $\mathbbm{1}_A^{C_1} = 0$, so that $A \subset X \setminus \operatorname{dom}_1(C_1)$, contradicting the initial assumption. This establishes that 
		$$
			\operatorname{dom}_1(C_1) \subset E^X_{\text{disc}}[\pi_1(\operatorname{Supp}(g))].
		$$
		Since the reverse inclusion always holds, we conclude that
		$$
			\operatorname{dom}_1(C_1) \asymp \pi_1(\operatorname{Supp}(g)),
		$$
		and thus, $g$ is coarsely full.
	\end{proof}
	
	The set of coarsely full operators is not limited to invertible ones. For instance, if $\mathcal{X}$ is a bounded coarse space, then any non-zero bounded operator between two coarse $\mathcal{X}$-modules is coarsely full. Our objective is to demonstrate that two isomorphic coarse $\mathcal{X}$-modules in $\mathcal{A}(\mathcal{X})$ possess asymptotically equivalent $\kappa$-domains, where $\kappa$ is either an infinite cardinal or $\kappa = 1$. To this end, we first establish a technical lemma.
	
	\begin{lemma} \label{lemma: translations of kappa-domain}
		Let $\mathcal{X}$ be a LFCM space, and let $C^0$ and $C^1$ be coarse $\mathcal{X}$-modules. Suppose that $t \colon C^0 \to C^1$ is an invertible controlled propagation operator. Then, for every $E^X_{\text{disc}}$-bounded measurable subset $B$ of $X$ satisfying $\operatorname{rk}(\mathbbm{1}_B^{C_1}) \geq \kappa$, there exists an element $A$ of the discrete partition $\{A_i\}_{i \in I}$ of $\mathcal{X}$ such that
		$$
			\operatorname{rk}(\mathbbm{1}_{A}^{C_0}) \geq \kappa \quad \text{ and } \quad \mathbbm{1}_{B}^{C_1} t \mathbbm{1}_{A}^{C_0} \neq 0.
		$$
	\end{lemma}
	\begin{proof}
		Define $A' = E^X_{\text{disc}} \circ \operatorname{Supp}(t)^T[B]$. Since $t$ has controlled propagation, the set $\operatorname{Supp}(t)^T$ is an entourage, implying that $A'$ is a bounded, measurable subset of $X$. Observe that $\mathbbm{1}_B^{C_1} t \mathbbm{1}_{A'}^{C_0} = \mathbbm{1}_B^{C_1} t$, leading to the chain of inequalities
		$$
			\operatorname{rk}(\mathbbm{1}_{A'}^{C_0}) \ge \operatorname{rk}(\mathbbm{1}_B^{C_1} t \mathbbm{1}_{A'}^{C_0}) \ge \operatorname{rk}(\mathbbm{1}_B^{C_1} t) = \operatorname{rk}(\mathbbm{1}_B^{C_1}) \ge \kappa.
		$$
		Here, the final equality holds due to the invertibility of $t$. Since $A'$ is bounded, it is contained in a finite union of elements of the discrete partition $\{A_{i_k}\}_{k=1}^n$. Given that $\kappa$ is either an infinite cardinal or $\kappa = 1$, the inequality
		$$
			\kappa \le \operatorname{rk}(\mathbbm{1}_B^{C_1} t \mathbbm{1}_{A'}^{C_0}) \le \sum_{k=1}^n \operatorname{rk}(\mathbbm{1}_B^{C_1} t \mathbbm{1}_{A_{i_k}}^{C_0}) 
		$$
		implies that for some $1 \le k \le n$, the rank of $\mathbbm{1}_B^{C_1} t \mathbbm{1}_{A_{i_k}}^{C^0}$ must be at least $\kappa$. Since $t$ is invertible, we conclude
		$$
			\kappa \le \operatorname{rk}(\mathbbm{1}_B^{C_1} t \mathbbm{1}_{A_{i_k}}^{C_0}) \le \operatorname{rk}(t \mathbbm{1}_{A_{i_k}}^{C_0}) = \operatorname{rk}(\mathbbm{1}_{A_{i_k}}^{C_0}).
		$$
		Hence, setting $A = A_{i_k}$ satisfies the desired conditions, completing the proof of the lemma.
	\end{proof}
	
	We are now equipped to prove that the $\asymp$-equivalence class of the $\kappa$-domain is an invariant of a coarse module.
	\begin{theorem}\label{theorem: domain invariance}
		Let $\mathcal{X}$ be a LFCM space, and let $C^0$ and $C^1$ be coarse $\mathcal{X}$-modules. Suppose there exists an invertible approximable operator $t \colon C^0 \to C^1$. Then, for any infinite cardinal $\kappa$ or $\kappa = 1$, the $\kappa$-domains of $C^0$ and $C^1$ are asymptotically equivalent:
		$$
			\operatorname{dom}_{\kappa}(C^0) \asymp \operatorname{dom}_{\kappa}(C^1).
		$$
	\end{theorem}
	\begin{proof}
		First, we establish that if the $\kappa$-domain of $C^1$ is non-empty, then so is the $\kappa$-domain of $C^0$. For $\kappa = 1$, the condition $\operatorname{dom}_1(C^1) \neq \emptyset$ asserts that the underlying Hilbert space of $C^1$ is non-trivial. Since $t \colon C^0 \to C^1$ is invertible, it follows that the underlying Hilbert space of $C^0$ is also non-trivial, implying that $\operatorname{dom}_1(C^0) \neq \emptyset$. Now, let $\kappa$ be an infinite cardinal, and suppose there exists an $E^X_{\text{disc}}$-bounded subset $A \subset X$ such that $\operatorname{rk}(\mathbbm{1}_A^{C^1}) \geq \kappa$. As $t$ is an invertible approximable operator, we may find an invertible operator $t_0 \colon C^0 \to C^1$ with controlled propagation. Then, by Lemma \ref{lemma: translations of kappa-domain}, there exists an element $A$ in the discrete partition $\{A_i\}_{i \in I}$ of $\mathcal{X}$ such that $\operatorname{rk}(\mathbbm{1}_{A}^{C^0}) \ge \kappa$. Consequently, $\operatorname{dom}_\kappa(C^0) \neq \emptyset$.

		Let $C$ be a coarse $\mathcal{X}$-module with a non-empty $\kappa$-domain. Define a coarse $\mathcal{X}$-module $C_{\kappa}$ as follows: its underlying Hilbert space is $\mathbbm{1}^C_{\operatorname{dom}_{\kappa}(C)} H_C$, and its representation is given by $A \mapsto \mathbbm{1}^C_{A \cap \operatorname{dom}_{\kappa}(C)}$. It is straightforward to verify that $C_{\kappa}$ is a non-zero coarse $\mathcal{X}$-module. Furthermore, we observe that $\operatorname{dom}_{\kappa}(C) = \operatorname{dom}_1(C_{\kappa})$, and the inclusion $i_C \colon C_{\kappa} \to C$ is a controlled propagation isometry.

		Suppose we are given two coarse $\mathcal{X}$-modules $C^0$ and $C^1$ with non-empty $\kappa$-domains, along with an approximable invertible operator $t_0 \colon C^0 \to C^1$. Then, there exists an invertible bounded operator $t \colon C^0 \to C^1$ with controlled propagation. By Lemma \ref{lemma: Invertible operator is fully supported}, $t$ is coarsely full. Consider $\tilde{t} = i_{C^1}^* t i_{C^0} \colon C^0_{\kappa} \to C^1_{\kappa}$. As the inclusions $i_{C^i}$ have controlled propagation, so does $\tilde{t}$. Moreover, we have
		$$
			\operatorname{Supp}(\tilde{t}) = \operatorname{Supp}(t) \cap \operatorname{dom}_{1}(C^1_{\kappa}) \times \operatorname{dom}_1(C^0_{\kappa}).
		$$
		We shall now demonstrate that $\tilde{t}$ is coarsely full. For $\kappa = 1$, the operator $\tilde{t}$ coincides with $t$, which is coarsely full. For an infinite cardinal $\kappa$, let $B$ be an $E^X_{\text{disc}}$-bounded measurable subset of $X$ satisfying $\operatorname{rk}(\mathbbm{1}_B^{C^1}) \geq \kappa$. By Lemma \ref{lemma: translations of kappa-domain}, there exists an element $A$ in the discrete partition $\{A_i\}_{i \in I}$ such that
		$$
			\operatorname{rk}(\mathbbm{1}_A^{C^0}) \ge \kappa, \quad \text{and} \quad \mathbbm{1}_B^{C_1} t \mathbbm{1}_A^{C_0} \neq 0.
		$$
		It follows that $B \times A \subset \operatorname{Supp}(\tilde{t})$, and consequently, $\pi_1(\operatorname{Supp}(\tilde{t}))$ contains $B$. As $B$ was arbitrary, we conclude that $\tilde{t}$ is coarsely full. Therefore,
		$$
			\operatorname{dom}_{\kappa}(C^1) = \operatorname{dom}_1 (C^1_{\kappa}) \prec \operatorname{dom}_1 (C^0_{\kappa}) = \operatorname{dom}_{\kappa}(C^0).
		$$
		By symmetry, we obtain the converse subordination, and hence, the $\kappa$-domains of $C^0$ and $C^1$ are asymptotic.
	\end{proof}
	
	One may remark that the set $\{[\operatorname{dom}_{\kappa}(C)]\}_{\kappa}$ is not a complete invariant for coarse $\mathcal{X}$-modules. For example, let $\mathcal{X}$ be a LFCM space, and let $x, y \in X$ be points that belong to the same connected component of $\mathcal{X}$. Consider the following modules:
	
	\begin{enumerate}
    	\item $C_x$ is a bounded module from Example \ref{example: bounded module} built upon a point $x$;
    	\item $C$ is the direct sum of the bounded modules $C_x$ and $C_y$, built upon points $x, y \in \mathcal{X}$, respectively.
	\end{enumerate}

	The $\kappa$-domains of both $C$ and $C_x$ are empty for $\kappa \ge 2$. For $\kappa = 1$, the $1$-domain of $C_x$ consists of a single point $\{x\}$, while the $1$-domain of $C$ consists of two points $\{x, y\}$. Since $x$ and $y$ belong to the same connected component of $\mathcal{X}$, the $1$-domains of $C$ and $C_x$ are asymptotically equivalent, despite the underlying Hilbert spaces having different dimensions.
	
	\begin{lemma}
    	Let $\mathcal{X}$ be a LFCM space and $C_0, C_1$ be $\mathcal{X}$-modules. Then, for any cardinal $\kappa$, 
    	$$
        	\operatorname{dom}_{\kappa}(C_0) \subset \operatorname{dom}_{\kappa}(C_0 \oplus C_1).
    	$$
    	In particular, if $C_0$ is a faithful (respectively, ample) $\mathcal{X}$-module, then the direct sum $C_0 \oplus C_1$ is also a faithful (respectively, ample) $\mathcal{X}$-module.
	\end{lemma}
	\begin{proof}
    	The result follows from the observation that 
    	$$
        	\operatorname{rk}(\mathbbm{1}^{C_0 \oplus C_1}_A) = \operatorname{rk}(\mathbbm{1}_A^{C_0} \oplus \mathbbm{1}_A^{C_1}) \geq \operatorname{rk}(\mathbbm{1}^{C_0}_A),
    	$$
    	for any measurable subset $A$ of $X$.
	\end{proof}

	Observe that $\mathcal{A}_{\kappa}(\mathcal{X})$ does not necessarily contains a faithful $\mathcal{X}$-module. For instance, if $\mathcal{X}$ has $\aleph_1$ coarsely connected components, no coarse $\mathcal{X}$-module of rank $\aleph_0$ can be faithful. However, it is possible to fully characterise the cardinals for which the approximable category contains a faithful module.
	
	\begin{lemma} \label{lemma: cardinality and existence of faith/ample mods}
		Let $\mathcal{X}$ be a LFCM space with an infinite discrete partition. For any two discrete partitions $\{A_i\}_{i \in I}$ and $\{B_j\}_{j \in J}$, the cardinalities of $I$ and $J$ are equal. Moreover, for an infinite cardinal $\kappa$, the following statements are equivalent:
		\begin{enumerate}
    		\item $\kappa$ is greater than or equal to the cardinality of a discrete partition of $\mathcal{X}$;
    		\item $\mathcal{A}_{\kappa}(\mathcal{X})$ contains a faithful $\mathcal{X}$-module;
    		\item $\mathcal{A}_{\kappa}(\mathcal{X})$ contains an ample $\mathcal{X}$-module.
		\end{enumerate}
	\end{lemma}
	\begin{proof}
		Since $\{A_i\}_{i \in I}$ is a discrete partition of $\mathcal{X}$, for each $j \in J$, there exist only finitely many $i \in I$ such that $B_j \cap A_i \neq \emptyset$. Consequently, $|J| \leq |I| \times \aleph_0$. Similarly, one has $|I| \leq |J| \times \aleph_0$. Note that if $\mathcal{X}$ admits an infinite discrete partition, then any discrete partition of $\mathcal{X}$ must also be infinite. It follows that
		$$
			|I| = |I| \times \aleph_0 = |J| \times \aleph_0 = |J|.
		$$
		For the second part of the statement, note that every ample $\mathcal{X}$-module is necessarily faithful. Conversely, given a faithful $\mathcal{X}$-module $C$, let $H$ be a separable, infinite-dimensional Hilbert space. Define the $\mathcal{X}$-module $C \otimes H$, where the underlying Hilbert space is $H_C \otimes H$ and the representation $\mathbbm{1}_{\bullet}^{C \otimes H}$ is given by $\mathbbm{1}_A^{C \otimes H} = \mathbbm{1}_A^C \otimes \operatorname{id}_H$ for all subsets $A$. It is straightforward to verify that $C \otimes H$ is an ample $\mathcal{X}$-module with the same rank as $C$.
		
		It remains to establish the equivalence of (1) and (2). If $\kappa$ exceeds the cardinality of the discrete partition, then the module constructed in Example \ref{example: uniform module} is a faithful $\mathcal{X}$-module of rank $|I|$. Conversely, suppose $C$ is a faithful $\mathcal{X}$-module of rank $\kappa$. Define $I_C = \{i \in I \mid A_i \cap \operatorname{dom}_1(C) \neq \emptyset\} \subseteq I$. By discretising $\mathcal{X}$ and $\operatorname{dom}_1(C)$, there exists a gauge $E \in \mathcal{E}^I$ such that
		$$
			I = E[I_C] = \bigcup_{i \in I_C} E[\{i\}].
		$$
		It follows that the cardinality of $I$ is given by $|I_C| \times \sup_{i \in I_C} |E[\{i\}]|$. Since $|E[\{i\}]|$ is finite for all $i \in I$, we deduce $|I| = |I_C|$.
	\end{proof}
	
	We shall focus on the case of countably generated LFCM spaces. Accordingly, we aim to determine the cardinals $\kappa$ for which the approximable category $\mathcal{A}_{\kappa}(\mathcal{X})$ contains a faithful $\mathcal{X}$-module. The following lemma addresses this question in the connected case.
	
	\begin{lemma}\label{cardinality of discrete partition for countably generated LFCM spaces}
		If $\mathcal{X}$ is a countably generated LFCM space with a countable number of connected components, then $\mathcal{A}_{\aleph_0}(\mathcal{X})$ contains a faithful $\mathcal{X}$-module.
	\end{lemma}
	\begin{proof}
		Let $\{A_i\}_{i \in I}$ be a discrete partition of $\mathcal{X}$. By considering the discretisation $\mathcal{I} = (I, \mathcal{E}^I, \mathcal{P}(I))$ of $\mathcal{X}$, as in Proposition \ref{proposition: discritization of LFCM spaces}, we obtain a locally finite, countably generated coarse space. We may select a generating set $\{E_n\}_{n \in \mathbb{N}}$ for $\mathcal{I}$ satisfying the following conditions:
		\begin{enumerate}
			\item For every $n \in \mathbb{N}$, the inclusion $E_n \subset E_{n+1}$ holds;
			\item For every $n \in \mathbb{N}$, the entourage $E_n$ is a gauge;
			\item For every entourage $F \in \mathcal{E}^I$, there exists $n \in \mathbb{N}$ such that $F \subset E_n$.
		\end{enumerate}
		Since $\mathcal{I}$ has a countable amount of connected components, there exists a sequence $\{i_k\}_{k \in \mathbb{N}} \subset I$ such that for every $j \in J$ there are $n,k \in \mathbb{N}$ such that $(i_k, j) \in E_n$. As $\mathcal{I}$ is locally finite, the sets $E_n[i_k]$ are finite for all $n,k \in \mathbb{N}$. Hence,  
		$$
			I = \bigcup_{k \in \mathbb{N}} \bigcup_{n \in \mathbb{N}} E_n[i_k]
		$$
		is at most countable. Consequently, $\mathcal{X}$ admits a countable discrete partition. By Lemma \ref{lemma: cardinality and existence of faith/ample mods}, the approximable category $\mathcal{A}_{\aleph_0}(\mathcal{X})$ contains a faithful $\mathcal{X}$-module.
	\end{proof}
	
	In the case where $\mathcal{X}$ has an uncountable number of connected components, the theorem above does not hold. For instance, consider an uncountable disjoint union of copies of $\mathbb{N}$, equipped with an extended metric
	$$
		d(n,m) = \begin{cases}
			|n-m|, & \text{if } n,m \text{ belong to the same copy of } \mathbb{N}; \\
			\infty, & \text{otherwise.}
		\end{cases}
	$$
	In this setting, the cardinal $\kappa$ must be at least the cardinality of the set of connected components of the space for the approximable category of rank $\kappa$ modules to contain a faithful module.
	
	Finally, the notion of a domain introduced in this section enables us to remove the faithfulness assumption in Theorem \ref{theorem: rigidity of $*$-isomorphisms} by restricting the domain and codomain of the resulting coarse equivalence to the faithfulness domains of the respective coarse modules.
	
	\begin{corollary} \label{corollary: Nonfaithful rigidity}
		Let $\mathcal{X}$, $\mathcal{Y}$ be countably generated LFCM spaces, and let $C_X$, $C_Y$ be coarse modules. Suppose $U \colon H_{C_X} \to H_{C_Y}$ is a unitary such that $\operatorname{Ad}_{U}$ restricts to a $*$-isomorphism between the algebras of approximable operators on $C_X$ and $C_Y$, respectively. Then, for any $\delta \in (0,1)$, there exist sufficiently large sets $F \in \mathcal{E}^Y$ and $E \in \mathcal{E}^X$ such that
		$$
    		f^U_{\delta, F, E} \colon \operatorname{dom}_1(C_X) \to \operatorname{dom}_1(C_Y)
		$$
		is a coarse equivalence, with inverse $f^{U^*}_{\delta, E, F}$.
	\end{corollary}
	\begin{proof}
		Denote the $1$-domain of $C_X$ by $\mathcal{X}_0$, and the $1$-domain of $C_Y$ by $\mathcal{Y}_0$. It follows that $\mathcal{X}_0$ and $\mathcal{Y}_0$ are LFCM spaces with coarse and measurable structures induced from $\mathcal{X}$ and $\mathcal{Y}$, respectively. The modules $C_X$ and $C_Y$ are naturally $\mathcal{X}_0$- and $\mathcal{Y}_0$-modules, respectively. Observe that
		$$
    		\operatorname{End}_{\mathcal{A}(\mathcal{X})}(C_X) = \operatorname{End}_{\mathcal{A}(\mathcal{X}_0)}(C_X) \quad \text{ and } \quad \operatorname{End}_{\mathcal{A}(\mathcal{Y})}(C_Y) = \operatorname{End}_{\mathcal{A}(\mathcal{Y}_0)}(C_Y),
		$$
		since the support of a bounded operator for the original space and the $1$-domain coincide. Now, $U$ is a unitary that restricts to a $*$-isomorphism between algebras of approximable operators on faithful coarse modules, and by Theorem \ref{theorem: rigidity of $*$-isomorphisms}, it induces a coarse equivalence $f^U_{\delta, F, E} \colon \operatorname{dom}_1(C_X) \to \operatorname{dom}_1(C_Y)$.
	\end{proof}
	
	\section{Rigidity of fully faithful $*$-functors}
	
	Suppose $\mathcal{X}$ and $\mathcal{Y}$ are countably generated LFCM spaces, and let $F \colon \mathcal{A}(\mathcal{X}) \to \mathcal{A}(\mathcal{Y})$ be a fully faithful $*$-functor. Suppose that $\mathcal{A}(\mathcal{X})$ contains a faithful $\mathcal{X}$-module. For any faithful $\mathcal{X}$-module $C$, the functor $F$ induces a $*$-isomorphism 
	$$
		\phi_C \colon \operatorname{End}_{\mathcal{A}(\mathcal{X})}(C) \to \operatorname{End}_{\mathcal{A}(\mathcal{Y})}(F(C)).
	$$
	By Theorem~\ref{theorem: rigidity of $*$-isomorphisms}, there exists a unitary $U(C) \colon H_C \to H_{F(C)}$ that implements $\phi_C$. Furthermore, by Corollary~\ref{corollary: Nonfaithful rigidity}, the unitary $U(C)$ induces a coarse embedding
	
	\begin{equation} \label{equation: approximate relations of unitary implementing isos}
    	f^{U(C)}_{\delta, F, E} \colon \mathcal{X} \to \operatorname{dom}_1(F(C)),
	\end{equation}
	
	for any $\delta \in (0,1)$ and sufficiently large gauges $F \in \mathcal{E}^\mathcal{Y}$ and $E \in \mathcal{E}^\mathcal{X}$. This section examines the interplay between the approximate relations introduced in \eqref{equation: approximate relations of unitary implementing isos} under different parameter choices.
	
	\begin{lemma} \label{lemma: properties of approximate relations}
		Let $\mathcal{X}$ and $\mathcal{Y}$ be LFCM spaces, $C_X$ a faithful $\mathcal{X}$-module, and $C_Y$ a $\mathcal{Y}$-module. Suppose $T \colon H_{C_X} \to H_{C_Y}$ is a bounded operator, $u \in \operatorname{End}_{\mathcal{A}(\mathcal{X})}(C_X)$ and $v \in \operatorname{End}_{\mathcal{A}(\mathcal{Y})}(C_Y)$ are central unitaries. Then $f^{vTu}_{\delta, F, E} = f^T_{\delta, F, E}$, for any $\delta > 0$, and any entourages $F \in \mathcal{E}^Y$ and $E \in \mathcal{E}^X$.
	\end{lemma}
	\begin{proof}
		Observe that for any measurable subsets $A \in \mathfrak{A}^X$ and $B \in \mathfrak{A}^Y$, the following equality holds:
		$$
			\|\mathbbm{1}^{C_Y}_B vTu \mathbbm{1}^{C_X}_A\| = \|v \mathbbm{1}^{C_Y}_B T \mathbbm{1}^{C_X}_A u\| = \|\mathbbm{1}^{C_Y}_B T \mathbbm{1}^{C_X}_A\|.
		$$
		This equality arises from the fact that the central unitaries $u$ and $v$ commute with the projections $\mathbbm{1}_B^{C_Y}$ and $\mathbbm{1}_A^{C_X}$. Consequently, for any $\delta > 0$, and any entourages $F \in \mathcal{E}^Y$ and $E \in \mathcal{E}^X$, the approximate relations are identical.
	\end{proof}
	
	Lemma \ref{lemma: properties of approximate relations} ensures that the relations in \eqref{equation: approximate relations of unitary implementing isos} are independent of the choice of the unitary $U(C)$. Consider two sets of parameters, $(\delta_1, F_1, E_1)$ and $(\delta_2, F_2, E_2)$, for which the respective approximate relations are coarse embeddings. Define $\delta = \min(\delta_1, \delta_2)$, $F = F_1 \cup F_2$, and $E = E_1 \cup E_2$. By Lemma \ref{lemma: properties of approximate relations}, one has that
	$$
		f^{U(C)}_{\delta_1, F_1, E_1} \cup f^{U(C)}_{\delta_2, F_2, E_2} \subset f^{U(C)}_{\delta, F, E}.
	$$
	Enlarging $E$ and $F$ if necessary, one can assume that the relation on the right-hand side is also a coarse embedding. Applying Lemma \ref{lemma: included relations => asymptotic}, this inclusion yields the asymptotic equivalence
	$$
		f^{U(C)}_{\delta_1, F_1, E_1} \asymp f^{U(C)}_{\delta, F, E} \asymp f^{U(C)}_{\delta_2, F_2, E_2}.
	$$
	Hence, the relation (\ref{equation: approximate relations of unitary implementing isos}) is independent of the choices of $U(C)$, $\delta$, $F$, and $E$. We still have the freedom to choose a different faithful $\mathcal{X}$-module $C$.
	
	\begin{theorem} \label{*-functor => coarse emb}
		Let $\mathcal{X}$ and $\mathcal{Y}$ be countably generated LFCM spaces, and let $F \colon \mathcal{A}(\mathcal{X}) \to \mathcal{A}(\mathcal{Y})$ be a fully faithful $*$-functor. Suppose that $\mathcal{A}(\mathcal{X})$ and $\mathcal{A}(\mathcal{Y})$ each contain a faithful $\mathcal{X}$-module and a faithful $\mathcal{Y}$-module, respectively. Then there exists a coarse embedding $f^F \colon \mathcal{X} \to \mathcal{Y}$, unique up to closeness, such that for any faithful $\mathcal{X}$-module $C$, $\delta > 0$, and sufficiently large entourages $F \in \mathcal{E}^Y$ and $E \in \mathcal{E}^X$, the approximate relation $f^{U(C)}_{\delta, F, E}$ is close to $f^F$. Furthermore, if $F$ is a $*$-equivalence of categories, then $f^F$ is a coarse equivalence. 
	\end{theorem}
	\begin{proof}
		By Lemma \ref{lemma: properties of approximate relations} and Remark \ref{remark:: folded approx relations when folded param}, for every faithful coarse $\mathcal{X}$-module $C$, the approximate relations are independent of the choice of $\delta > 0$ and sufficiently large entourages $F \in \mathcal{E}^Y$ and $E \in \mathcal{E}^X$. Suppose two faithful $\mathcal{X}$-modules, $C$ and $D$, are given. The unitary $U(C) \oplus U(D) \colon C \oplus D \to F(C) \oplus F(D)$ induces an isomorphism $\operatorname{Ad}_{U(C) \oplus U(D)}$ between the algebras of approximable operators on $C \oplus D$ and $F(C) \oplus F(D)$, respectively. It follows that $U(C) \oplus U(D)$ is coarse-like, and its approximate relations are controlled. Observe that for any measurable bounded sets $B \in \mathfrak{A}^Y$ and $A \in \mathfrak{A}^X$, one has
		$$
			\|\mathbbm{1}^{F(C) \oplus F(D)}_B U(C) \oplus U(D) \mathbbm{1}_{A}^{C \oplus D}\| \ge \|\mathbbm{1}_B^{F(C)} U(C) \mathbbm{1}_A^C\|.
		$$
		It follows that the approximate relations corresponding to $U(C)$ and $U(C) \oplus U(D)$ are contained in each other. By Lemma \ref{lemma: included relations => asymptotic}, these relations are asymptotic. Similarly, it can be demonstrated that the approximate relations of $U(D)$ and $U(C) \oplus U(D)$ are asymptotic. Therefore, one obtains the equivalences
		$$
			f^{U(C)}_{\delta, F, E} \asymp f^{U(C) \oplus U(D)}_{\delta, F, E} \asymp f^{U(D)}_{\delta, F, E}.
		$$
		Hence, the class of the approximate relation is independent of the choice of a faithful coarse $\mathcal{X}$-module. If $F$ is a $*$-equivalence of categories, it is essentially surjective. In particular, there exists a coarse $\mathcal{X}$-module $D$ such that $F(D)$ is a faithful $\mathcal{Y}$-module. For any faithful $\mathcal{X}$-module $C$, the module $C \oplus D$ is also faithful. Moreover, its image $F(C \oplus D)$ is isomorphic to $F(C) \oplus F(D)$, which remains faithful since $F(D)$ is faithful. By Theorem \ref{theorem: rigidity of $*$-isomorphisms}, it follows that the approximate relations corresponding to $U(C) \oplus U(D)$ are coarse equivalences.
	\end{proof}
	
	Let $\mathcal{X}$ and $\mathcal{Y}$ be countably generated LFCM spaces. Suppose a family $\{C_i\}_{i \in I}$ of faithful $\mathcal{X}$-modules and a family $\{D_i\}_{i \in I}$ of $\mathcal{Y}$-modules are given. In particular, the preceding theorem establishes that if a family of $*$-isomorphisms  
	$$
		\phi_i \colon \operatorname{End}_{\mathcal{A}(\mathcal{X})}(C_i) \to \operatorname{End}_{\mathcal{A}(\mathcal{Y})}(D_i)
	$$  
	assembles into a full and faithful $*$-functor $F \colon \mathcal{A}(\mathcal{X}) \to \mathcal{A}(\mathcal{Y})$, then these $*$-isomorphisms induce the same coarse embedding, unique up to closeness, between $\mathcal{X}$ and $\mathcal{Y}$.
	
	\begin{corollary}
		Let $\mathcal{X}$ and $\mathcal{Y}$ be countably generated LFCM spaces, and let $F \colon \mathcal{A}(\mathcal{X}) \to \mathcal{A}(\mathcal{Y})$ be a $*$-equivalence of categories. Suppose that $\mathcal{A}(\mathcal{X})$ and $\mathcal{A}(\mathcal{Y})$ each contain a faithful $\mathcal{X}$-module and a faithful $\mathcal{Y}$-module, respectively. Then, for any faithful $\mathcal{X}$-module $C$, the $\mathcal{Y}$-module $F(C)$ is also faithful.
	\end{corollary}
	\begin{proof}
		The closeness class of asymptotic relations is independent of the choice of a faithful $\mathcal{X}$-module $C$. For two faithful $\mathcal{X}$-modules $C$ and $D$ let $f^C$ and $f^D$ denote the approximate relations corresponding to $U(C)$ and $U(D)$. One has
		$$
			\operatorname{dom}_0(F(C)) \asymp \pi_2(f^C) \asymp \pi_2(f^D) \asymp \operatorname{dom}_0(F(D)).
		$$
		Since there exists a faithful $\mathcal{X}$-module $C$ whose image $F(C)$ is a faithful $\mathcal{X}$-module, the statement follows.
	\end{proof}
	
	In the case where $\mathcal{X}$ has a countable number of connected components, Lemma \ref{cardinality of discrete partition for countably generated LFCM spaces} enables us to dispense with the requirement that approximable categories contain faithful modules, as the approximable category of separable coarse modules already includes a faithful module.
	
	\begin{corollary} \label{ff *-funct => coarse emb (2)}
		Let $\mathcal{X}$ and $\mathcal{Y}$ be countably generated LFCM spaces with at most countably many connected components, and let $F \colon \mathcal{A}_{\aleph_0}(\mathcal{X}) \to \mathcal{A}_{\aleph_0}(\mathcal{Y})$ be a fully faithful $*$-functor. Then there exists a coarse embedding $f^F \colon \mathcal{X} \to \mathcal{Y}$, unique up to closeness, such that for any faithful $\mathcal{X}$-module $C$, $\delta > 0$, and sufficiently large entourages $F \in \mathcal{E}^Y$ and $E \in \mathcal{E}^X$, the approximate relation $f^{U(C)}_{\delta, F, E}$ is close to $f^F$. Furthermore, if $F$ is a $*$-equivalence of categories, then $f^F$ is a coarse equivalence.
	\end{corollary}
	
	\section{Structure of fully faithful $*$-functors}
	
	The previous section concluded that any full and faithful $*$-functor $F \colon \mathcal{A}(\mathcal{X}) \to \mathcal{A}(\mathcal{Y})$ between approximable categories of countably generated LFCM spaces that contain faithful modules induces a unique up to closeness coarse embedding $f^F \colon \mathcal{X} \to \mathcal{Y}$. In this section, we investigate the relation between $F$ and a $*$-functor $f^F_*$ induced by $f^F$. Recall (\cite[Chapter 8, section 3]{Richter_2020}) that an additive functor $F$ induces a natural isomorphism $\alpha$ whose components are given by
	$$  
		\alpha_{C,D} \colon F(C) \oplus F(D) \to F(C \oplus D), \qquad \alpha_{C,D}= F(i_C) \pi_{F(C)} + F(i_D) \pi_{F(D)}
	$$  
	for all objects $C$ and $D$. We can express the components of $\alpha$ in terms of unitaries that implement the $*$-isomorphisms between algebras of approximable operators.

	\begin{lemma} \label{lemma: additivity nat.iso => obstruction to unitaries}
		Let $\mathcal{X}$ and $\mathcal{Y}$ be LFCM spaces, $F \colon \mathcal{A}(\mathcal{X}) \to \mathcal{A}(\mathcal{Y})$ be a fully faithful $*$-functor and $C, D$ be coarse $\mathcal{X}$-modules. Let $\alpha$ be a natural isomorphism provided by additivity of $F$, then there are central unitaries $u \in \operatorname{End}_{\mathcal{A}(\mathcal{X})}(C)$, and $v \in \operatorname{End}_{\mathcal{A}(\mathcal{X})}(D)$ such that
		$$
			\alpha_{C, D} = U(C \oplus D) (U(C) \oplus U(D))^* (u \oplus v).
		$$
	\end{lemma}
	\begin{proof}
		Since $\alpha$ is a natural transformation, for any approximable operators $t \colon C \to C'$, and $h \colon D \to D'$ the following diagram commutes
		$$
			\begin{tikzcd}
				F(C') \oplus F(D') \arrow{d}{\alpha_{C,D}} \arrow{rr}{F(t) \oplus F(h)} && F(C') \oplus F(D') \arrow{d}{\alpha_{C',D'}} \\
				F(C \oplus D)  \arrow{rr}{F(t \oplus h)} && F(C' \oplus D').
			\end{tikzcd}
		$$
		Take $C = C'$, $D = D'$, and define the unitary $V$ as the composition
		$$
			\begin{tikzcd}
				V \colon C \oplus D \arrow{rr}{U(C) \oplus U(D)} && F(C) \oplus F(D) \arrow{r}{\alpha_{C, D}} & F(C \oplus D) \arrow{rr}{U(C \oplus D)^*} && C \oplus D.
			\end{tikzcd}
		$$
		It is straightforward from the diagram that $V$ satisfies $\operatorname{Ad}_V(t \oplus h) = t \oplus h$. It follows that the diagonal entries of the matrix representation of $V$ for the decomposition $C \oplus D$ are central unitaries in $\operatorname{End}_{\mathcal{A}(\mathcal{X})}(C)$ and $\operatorname{End}_{\mathcal{A}(\mathcal{X})}(D)$ respectively. Simple computations show that the off-diagonal entries of $V$ are zeros. Therefore $V$ is a diagonal matrix whose entries are central unitaries, but then
		$$
			\alpha_{C,D} = U(C \oplus D) V (U(C) \oplus U(D))^*.
		$$
		The equality up to central unitaries follows.
	\end{proof}
	
	Generally, the morphisms $\alpha_{C, D}$ need not be of controlled propagation, as the following example shows.
	
	\begin{example}
		Let $\mathcal{X}$ be a LFCM space. For every coarse $\mathcal{X}$-module $C$, pick an approximable unitary $U(C) \colon C \to C$. Define a functor $F \colon \mathcal{A}(\mathcal{X}) \to \mathcal{A}(\mathcal{X})$ as identity on objects, and for every approximable operator $t \colon C \to D$ define $F(t) \colon C \to D$ as
		$$
			F(t) = U(D) t U(C)^*
		$$
		By Lemma \ref{lemma: additivity nat.iso => obstruction to unitaries} the components of the natural isomorphism $\alpha \colon F \oplus F \Rightarrow F$ are given by 
		$$
			\alpha_{C,D} = U(C \oplus D)(U(C) \oplus U(D))^*.
		$$
		In view of the arbitrary choice of the unitaries $U(C)$, the components of $\alpha$ need not be of controlled propagation.
	\end{example}
	
	The general form of full and faithful $*$-functors closely resembles the previous example. In what follows, a central unitary $u$ in a unital $C^*$-algebra $A$ refers to a unitary element that belongs to the centre of $A$.
	
	\begin{theorem} \label{theorem: structure of fully faithful $*$-functors}
		Let $\mathcal{X}$ and $\mathcal{Y}$ be LFCM spaces, $F \colon \mathcal{A}(\mathcal{X}) \to \mathcal{A}(\mathcal{Y})$ be a full and faithful $*$-functor. For coarse $\mathcal{X}$-module $C$ let $U(C) \colon H_C \to H_{F(C)}$ denote the unitary that implements the $*$-isomorphism between algebras of approximable operators on $C$ and $F(C)$ induced by $F$. Then for every pair of $\mathcal{X}$-modules $C$ and $D$ and an approximable operator $t \colon C \to D$ there are central unitaries $u \in \operatorname{End}_{\mathcal{A}(\mathcal{X})}(C)$, and $v \in \operatorname{End}_{\mathcal{A}(\mathcal{X})}(D)$ such that
		$$
			F(t) = U(D)^* v^* t u^* U(C).
		$$
	\end{theorem}
	\begin{proof}
		Let $t \colon C \to D$ be an approximable operator. It can be seen as an operator $\hat{t} \colon C \oplus D \to C \oplus D$ given by $\hat{t} = i_D t \pi_C$. Note that
		$$
			\begin{aligned}
				F(t) 	& =  F(\pi_D) F(\hat{t}) F(i_C) \\
						& = F(\pi_D) U(C \oplus D) \hat{t} U(C \oplus D)^* F(i_C) \\
					 	& \approx  F(\pi_D) \alpha_{C,D} (U(C) \oplus U(D)) \hat{t} (U(C) \oplus U(D))^* \alpha_{C,D}^* F(i_C) \\
					 	& = \pi_{F(D)} (U(C) \oplus U(D)) \hat{t} (U(C) \oplus U(D))^* i_{F(C)} \\
					 	& = U(D) t U(C)^*,
			\end{aligned}
		$$
		where $\approx$ denotes equality up to central unitaries. The conclusion follows.
	\end{proof}
	
	\begin{example}
		Let $(X,d)$ be a metric space. Since its metric coarse structure $\mathcal{E}_d$ is coarsely connected, it comprises a single connected component. Let $C$ be a coarse $\mathcal{X}$-module. By the result of \cite{MVmodules}, the algebra of approximable operators $\operatorname{End}_{\mathcal{A}(\mathcal{X})}(C)$ contains the algebra of compact operators on the underlying Hilbert space $H_C$. Consequently, the only central unitaries in $\operatorname{End}_{\mathcal{A}(\mathcal{X})}(C)$ are of the form $\alpha 1$ for some $\alpha$ with $|\alpha| = 1$. For a non-connected LFCM space $\mathcal{X}$ with connected components $\{X_i\}_{i \in J}$, any central unitary $u \in \operatorname{End}_{\mathcal{A}(\mathcal{X})}(C)$ is given by  
		$$
			\operatorname{SOT-} \sum_{i \in J} \alpha_i \mathbbm{1}_{X_i},
		$$  
		for some complex numbers $\{\alpha_i\}_{i \in J}$ of modulus one.

	\end{example}
	
	Theorem \ref{theorem: structure of fully faithful $*$-functors} and Lemma \ref{lemma: additivity nat.iso => obstruction to unitaries} force us to deal with morphisms modulo central unitaries. We introduce certain decorative objects to make the results more readable.
	
	\begin{definition}
		Let $\mathcal{X}$, $\mathcal{Y}$ be LFCM spaces, $C_X$ and $C_Y$ be coarse modules. Two bounded operators $t, s \colon H_{C_X} \to H_{C_Y}$ are said to be \emph{equivalent modulo central unitaries} (in symbols, $t \cong_{c} s$) if there are central unitaries $u \in \operatorname{End}_{\mathcal{A}(\mathcal{X})}(C_X)$ and $v \in \operatorname{End}_{\mathcal{A}(\mathcal{Y})}(C_Y)$ such that $t = vsu$.
	\end{definition}
	
	It is easy to see that for any coarse $\mathcal{X}$-modules $C$ and $D$, the relation $\cong_u$ is an equivalence relation on the set of approximable operators from $C$ to $D$. More can be said, recall (\cite[Chapter 2, section 8]{maclane}) that an equivalence relation on the Hom-sets of a category is said to be a congruency if it respects the composition.
	
	\begin{lemma}
		The relation $\cong_u$ is a congruency on the approximable category
	\end{lemma}
	\begin{proof}
		It is enough to show that for any approximable operator $t \colon C \to D$ and a central unitary $u \colon C \to C$, there is a central unitary $v \colon D \to D$ such that $tu = vt$. As $t$ is approximable, there is a controlled propagation operator $s \colon C \to D$ such that $\|t - s\| < \varepsilon$. Let $\{X_i\}_{i \in I}$ be a decomposition of $\mathcal{X}$ into its connected component. Each $X_i$ is measurable as it is a disjoint union of elements of the discrete partition. As $s$ is of controlled propagation, one has
		$$
			\mathbbm{1}^D_{X_i} s = s \mathbbm{1}^C_{X_i}.
		$$
		Let $u \colon C \to C$ be a central unitary, then $u = \text{SOT-}\sum_i \alpha_i \mathbbm{1}_{X_i}^C$, for some complex numbers $\alpha_i$ of modulus one. Define a central unitary in $\operatorname{End}_{\mathcal{A}(\mathcal{X})}(D)$ as $v = \text{SOT-}\sum_i \alpha_i \mathbbm{1}_{X_i}^D$. It follows that 
		$$
			su = \text{SOT-}\sum_{i} s \mathbbm{1}_{X_i}^C u = \text{SOT-}\sum_{i} \alpha_i \mathbbm{1}_{X_i}^D s = vs.
		$$
		By tending $\varepsilon$ to $0$ one concludes that $vt = tu$. Now suppose that $t_1 \cong_u t_2$ and $s_1 \cong_u s_2$ for some approximable operators $t_1, t_2 \colon C \to D$, and $s_1, s_2 \colon D \to G$, then $s_1 t_1 = v_1 s_2 u_1 v_2 t_2 u_2$, for some central unitaries $v_1 \colon G \to G$, $u_1, v_2 \colon D \to D$, and $u_2 \colon C \to C$. From the previous claim, we deduce that $s_1 t_1 = u s_2 t_2 v$ for some central unitaries $u \colon G \to G$, and $v \colon C \to C$.
	\end{proof}
	
	Consider a quotient category $\mathcal{A}(\mathcal{X}) / \cong_u$ whose objects are coarse $\mathcal{X}$-modules and whose morphisms are classes of $\cong_u$-equivalence of approximable operators. Denote by $\Pi \colon \mathcal{A}(\mathcal{X}) \to \mathcal{A}(\mathcal{X}) / \cong_u$ the canonical quotient functor. Note that the quotient category is not a $C^*$-category anymore as $\cong_u$ does not preserve sums. Since a full and faithful $*$-functor $F \colon \mathcal{A}(\mathcal{X}) \to \mathcal{A}(\mathcal{Y})$ maps central unitaries to central unitaries, it follows that $\Pi \circ F$ preserves the conjugacy relation. Consequently, there is a functor $\tilde{F} \colon \mathcal{A}(\mathcal{X}) / \cong_u \to \mathcal{A}(\mathcal{Y}) / \cong_u$, which acts identically on objects and for every approximable operator $t$, the functor $\tilde{F}$ sends the class $[t]$ to $[F(t)]$. We will refer to $\tilde{F}$ as a decent of $F$.
	
	\begin{definition}
		For two full and faithful $*$-functors $F, G \colon \mathcal{A}(\mathcal{X}) \to \mathcal{A}(\mathcal{Y})$, a \emph{natural transformation modulo central unitaries} from $F$ to $G$ is a natural isomorphism $\eta \colon \tilde{F} \Rightarrow \tilde{G}$ between their decents. 
	\end{definition}
	
	Suppose given two full and faithful $*$-functors $F, G \colon \mathcal{A}(\mathcal{X}) \to \mathcal{A}(\mathcal{Y})$, and a collection of approximable unitaries $\eta_C \colon F(C) \to G(C)$ for every object $C$ of $\mathcal{A}(\mathcal{X})$ such that for every pair of coarse $\mathcal{X}$-modules $C$, and $D$ there are central unitaries $u_C \colon C \to C$ and $u_D \colon D \to D$ such that the following diagram commutes for every approximable operator $t \colon C \to D$:
	$$
		\begin{tikzcd}
			F(C) \arrow{d}{\eta_C} \arrow{rr}{F(u_Dtu_C)} & & F(D) \arrow{d}{\eta_D} \\
			G(C) \arrow{rr}{G(t)} & & G(D).
		\end{tikzcd}
	$$
	It is easy to check that the collection of $\cong_u$-equivalence classes $[\eta_C]$ defines a natural modulo central unitaries isomorphism $\eta \colon \tilde{F} \Rightarrow \tilde{G}$. We will show that for a full and faithful $*$-functor $F \colon \mathcal{A}(\mathcal{X}) \to \mathcal{A}(\mathcal{Y})$ the $*$-functor $f^F_*$ induced by a coarse embedding $f^F$ from Theorem \ref{*-functor => coarse emb} is naturally modulo central unitaries isomorphic to $F$. To proceed, we need a strengthening of Theorem \ref{theorem: Iso => unitary is coarse-like}.
	
	\begin{theorem}[\cite{MVrigidity}, Proposition 9.11]
		Let $\mathcal{X}$ and $\mathcal{Y}$ be LFCM spaces, $C_X$ and $C_Y$ be ample coarse modules. Let $U \colon H_{C_X} \to H_{C_Y}$ be a unitary that implements a $*$-isomorphism
		$$
			\phi \colon \operatorname{End}_{\mathcal{A}(\mathcal{X})}(C_X) \to \operatorname{End}_{\mathcal{A}(\mathcal{Y})}(C_Y).
		$$
		Then for every $\varepsilon > 0$ there exists a controlled relation $R \subset Y \times X$, and an operator $S \colon H_{C_X} \to H_{C_Y}$ supported on $R$ such that $\|U - S\| < \varepsilon$.
	\end{theorem}
	
	The enhancement provided by the above theorem comes at the cost of the ampleness requirement. This is not problematic in our setting due to Lemma \ref{lemma: cardinality and existence of faith/ample mods}, as demonstrated in the proof of the subsequent theorem. However, this ampleness condition may become restrictive if one considers alternative variants of approximable categories. For instance, Lemma \ref{lemma: cardinality and existence of faith/ample mods} does not apply to categories of locally finite coarse modules, i.e., coarse modules where the projections $\mathbbm{1}_A$ have finite rank for bounded subsets $A$.
	
	\begin{theorem} \label{theorem: Structure_of_faf_star_functors}
		Let $\mathcal{X}$ and $\mathcal{Y}$ be LFCM spaces, $F \colon \mathcal{A}(\mathcal{X}) \to \mathcal{A}(\mathcal{Y})$ be a full and faithful $*$-functor. Suppose that $\mathcal{A}(\mathcal{X})$ contains a faithful $\mathcal{X}$-module. Let $f^F \colon \mathcal{X} \to \mathcal{Y}$ be a coarse embedding induced by $F$, then $F$ is naturally modulo central unitaries isomorphic to $f^F_*$.
	\end{theorem}
	\begin{proof}
		Fix a measurable coarse embedding $f \colon \mathcal{X} \to \mathcal{Y}$ that represents $f^F$. For a coarse $\mathcal{X}$-module $C$ we have two isomorphisms
		$$
			\phi_C, \psi_C \colon \operatorname{End}_{\mathcal{A}(\mathcal{X})}(C) \to \operatorname{End}_{\mathcal{A}(\mathcal{Y})}(F(C)),
		$$
		where $\phi_C$ is induced by $F$, and $\psi_C$ is induced by $f_*$. Let $U(C)$, $V(C)$ be unitaries that implement $\phi_C$, $\psi_C$ resepctivelly. Consider a unitary $\eta_C = U(C)V(C)^*$. The following diagram commutes modulo central unitaries for every approximable operator $t \colon C \to D$:
		$$
			\begin{tikzcd}
				F(C) \arrow{r}{F(t)} & F(D) \\
				f_*(C) \arrow{u}{\eta_C} \arrow{r}{f_*(t)}& f_*(D) \arrow{u}{\eta_D}.
			\end{tikzcd}
		$$
		It remains to check that $\eta_C$ are approximable, since then the classes of $\cong_u$-equivalences of $\eta_C$ will define a natural modulo central unitaries isomorphism. Note that for any measurable subsets $A, B$ of $Y$, one has an equality
		$$
			\|\mathbbm{1}_B^{F(C)} \eta_C \mathbbm{1}_A^{f_*(C)}\| = \|\mathbbm{1}_B^{F(C)} U(C) V(C)^* \mathbbm{1}_A^{f_*(C)} V(C)\| = \|\mathbbm{1}_B^{F(C)} U(C) \mathbbm{1}_{f^{-1}(A)}^C\|.
		$$
		Therefore, for every $\delta > 0$, $F, E \in \mathcal{E}^Y$ there is an inclusion of approximate relations
		$$
			f^{\eta_C}_{\delta, F, E} \subset f^{U(C)}_{\delta, F, (f \times f)^{-1}(E)} \circ f^T.
		$$
		By the construction of $f^F$ from Theorem \ref{theorem: rigidity of $*$-isomorphisms}, the approximate relations corresponding to $U(C)$ are subordinate to $f$. It follows that the approximate relations corresponding to $\eta_C$ are entourages in $\mathcal{X}$. Suppose that $C$ is an ample $\mathcal{X}$-module. Then the for every $\varepsilon > 0$ there exists a controlled relation $R \subset Y \times Y$, and an operator $s \colon C \to F(C)$ supported on $R$ such that 
		$$
			\|s V(C)^* - \eta_C\| = \|s - U(C)\| < \varepsilon.
		$$
		Since $R$ is a controlled relation, and $f$ is a coarse embedding, we deduce that $R \circ E^X_{\text{disc}} \circ f^T$ is a controlled relation, in particular, the support of $s V(C)^*$ is a controlled relation. It follows that
		$$
			f^{\eta_C}_{\delta, F, E} \prec \operatorname{Supp}(s V(C)^*).
		$$
		By picking large enough gauges $F, E$, the approximate relations of $\eta_C$ become densely defined. It follows from Lemma \ref{lemma: included relations => asymptotic} that $sV(C)^*$ has controlled propagation. Thus, $\eta_C$ is approximable. Now for arbitrary coarse $\mathcal{X}$-module $C$, pick an ample coarse $\mathcal{X}$-module $\tilde{C}$ such that there is an approximable isometry $t \colon C \to \tilde{C}$. The following diagram commutes up to central unitaries:
		$$
			\begin{tikzcd}
				F(C) \arrow{r}{F(t)} & F(\tilde{C}) \\
				f_*(C) \arrow{u}{\eta_C} \arrow{r}{f_*(t)}& f_*(\tilde{C}) \arrow{u}{\eta_{\tilde{C}}}.
			\end{tikzcd}
		$$
		It follows that up to central unitaries $\eta_C$ is equal to $F(t)^* \eta_{\tilde{C}} f_*(t)$, and the latter is an approximable operator. Therefore, $\eta_C$ is approximable.
	\end{proof}
	
	Note that the only reason to introduce natural isomorphisms modulo central unitaries is the failure of commutativity of the squares
	$$
		\begin{tikzcd}
			F(C) \arrow{r}{F(t)} & F(D) \\
			C \arrow{u}{U(C)}\arrow{r}{G(t)} & D \arrow{u}{U(D)}.
		\end{tikzcd}
	$$
	Such squares always commute when $F$ is a $*$-functor induced by a coarse measurable map $f \colon \mathcal{X} \to \mathcal{Y}$. In particular, if one can pick unitaries $U(C)$ so that all the squares above commute, then one can show that $F$ and $f_*$ are naturally isomorphic.
	 
	\begin{remark}
		Theorem \ref{theorem: Structure_of_faf_star_functors} provides more than natural modulo central unitaries isomorphism. In fact, the choice of components $\eta_C$, $\eta_D$ depends only on the modules $C$ and $D$ and not on the morphism $t \colon C \to D$. 
	\end{remark}
	
	\section{Corollaries and open questions}
	
	In the final subsection, we summarise the consequences of the present work and outline certain questions that remain unanswered. We begin by revisiting the initial question of when a collection of $*$-isomorphisms between algebras of approximable operators on faithful coarse modules induces the same coarse equivalence up to closeness. 
	
	\begin{corollary} \label{corollary: initial goal}
		Let $\mathcal{X}$, $\mathcal{Y}$ be countably generated coarse spaces, $\{C_i\}_{i \in I}$ be a collection of faithful $\mathcal{X}$-modules, $\{D_i\}_{i \in I}$ be a collection of $\mathcal{Y}$-modules, and 
		$$
			\{\phi_i \colon \operatorname{End}_{\mathcal{A}(\mathcal{X})}(C_i) \to \operatorname{End}_{\mathcal{A}(\mathcal{Y})}(D_i)\}_{i \in I}
		$$
		be a collection of $*$-ismomorphisms. The following are equivalent:
		\begin{enumerate}
			\item The coarse embeddings induced by $\phi_i$ are asymptotic;
			\item There is a full and faithful $*$-functor $F \colon \mathcal{A}(\mathcal{X}) \to \mathcal{A}(\mathcal{Y})$ that induces the isomorphisms $\phi_i$.
		\end{enumerate}
	\end{corollary}
	\begin{proof}
		The implication (2. $\Rightarrow$ 1.) is established by Theorem \ref{*-functor => coarse emb}. Vise-versa, let $f$ be a measurable representative of the class of coarse embeddings induced by $\phi_i$. As $f$ is a coarse embedding, the induced $*$-functor $f_*$ is full and faithful. Let $U(C_i)$ be unitaries that implement $\phi_i$, and $V(C)$ denote the unitaries that implement $f_*$ at $C$. Define a new functor $F \colon \mathcal{A}(\mathcal{X}) \to \mathcal{A}(\mathcal{Y})$ as follows. On objects, if $C$ belongs to the collection $\{C_i\}_{i \in I}$, i.e $C = C_i$ for some $i \in I$, then $F(C) = D_i$, otherwise, $F(C) = f_*(C)$. For an approximable operator $t \colon C^0 \to C^1$ define
		\begin{enumerate}
			\item If $C^0$ and $C^1$ belong to the collection $\{C_i\}_{i \in I}$, then $F(t) = U(C^1) t U(C^0)^*$;
			\item If $C^0$ belongs to the collection $\{C_i\}_{i \in I}$, and $C^1$ does not, then $F(t) = V(C^1) t U(C^0)^*$;
			\item If $C^1$ belongs to the collection $\{C_i\}_{i \in I}$, and $C^0$ does not, then $F(t) = U(C^1) t V(C^0)^*$;
			\item Otherwise, $F(t) = V(C^1) t V(C^0)^*$.
		\end{enumerate}
		It is clear that $F$ defines a full and faithful $*$-functor. 
	\end{proof}
	
	Let $\textbf{LFCM}^{\aleph_0}_{\text{emb}}$ be a category whose objects are countably generated coarse spaces and whose morphisms are closeness classes of coarse embeddings. Let $\mathcal{C}/\cong_u$ be a category whose objects are approximable categories of coarse modules of rank $\aleph_0$ and whose morphisms are equivalence classes of full and faithful $*$-functors modulo natural up to central unitaries isomorphisms.
	
	\begin{corollary} \label{corollary: cat equiv}
		The functor $\mathcal{A}_{\aleph_0} \colon \textbf{LFCM}^{\aleph_0}_{\text{emb}} \to \mathcal{C}/\cong_u$ is an equivalence of categories.
	\end{corollary}
	\begin{proof}
		The fact that $\mathcal{A}_{\aleph_0}$ is full is established in Theorem \ref{theorem: Structure_of_faf_star_functors}. It is essentially surjective by definition of $\mathcal{C}$. It remains to show that it is faithful. Suppose that for two measurable coarse embeddings $f, g \colon \mathcal{X} \to \mathcal{Y}$ one has $f_* \cong_u g_*$. In particular, for every coarse $\mathcal{X}$-module $C$ there is an approximable unitary $\eta_c \colon f_*(C) \to g_*(C)$ such that for any approximable operator $t \colon C \to C$ the following diagram commutes in $\mathcal{A}(\mathcal{Y}) / \cong_u$:
		$$
			\begin{tikzcd}
				f_*(C) \arrow{r}{f_*(t)} \arrow{d}{\eta_c} & f_*(C) \arrow{d}{\eta_c}\\
				g_*(C) \arrow{r}{g_*(t)} & g_*(C).
			\end{tikzcd}
		$$
		It follows that $\operatorname{Ad}_{\eta_c} \colon \operatorname{End}_{\mathcal{A}(\mathcal{Y})}(f_*(C)) \to \operatorname{End}_{\mathcal{A}(\mathcal{Y})}(g_*(C))$ is a $*$-isomorphism, therefore, $\eta_c$ is a coarse-like approximable unitary. By Lemma \ref{proposition: approximable, coarse-like unitary => controlled}, $\eta_c$ is of controlled propagation. Consider unitaries
		$$
			U(C) \colon C \to f_*(C), \quad V(C) \colon C \to g_*(C),
		$$
		that implement $f_*$ and $g_*$ at $C$ respectively. Let $\nu_c = V(C)U(C)^*$. Note that
		$$
			\|\mathbbm{1}^{g_*(C)}_B \nu_c \mathbbm{1}_A^{f_*(C)}\| = \|\mathbbm{1}^{g_*(C)}_B \eta_c \mathbbm{1}_A^{f_*(C)}\|,
		$$
		thus $\operatorname{Supp}(\eta_c) = \operatorname{Supp}(\nu_c)$, but $\operatorname{Supp}(\nu_c) = (f \times g)(E^X_{\text{disc}})$. It follows that $(f \times g)(E^X_{\text{disc}})$ is an entourage, so $f \sim g$.
	\end{proof}
	
	One can similarly consider the quasi-local category, whose objects are coarse $\mathcal{X}$-modules and whose morphisms are quasi-local operators (see \cite[Definition 5.14]{MVmodules} for the definition of quasi-local operators). However, we did not manage to prove the analogue of Theorem \ref{theorem: domain invariance} for coarse $\mathcal{X}$-modules isomorphic via a quasi-local unitary.
	
	\begin{question}
		Let $\mathcal{X}$, $\mathcal{Y}$ be LFCM spaces. Suppose given a full and faithful $*$-functor $F \colon \operatorname{Ql}(\mathcal{X}) \to \operatorname{Ql}(\mathcal{Y})$ between the quasi-local categories of $\mathcal{X}$ and $\mathcal{Y}$. Do the coarse equivalences induced by $*$-isomorphisms of quasi-local algebras of faithful modules coincide up to closeness? 
	\end{question}

	
\end{document}